\RequirePackage{ifpdf}
\ifpdf % We are running pdfTeX in pdf mode
\documentclass[pdftex]{sigma}
\else
\documentclass{sigma}
\fi

\numberwithin{equation}{section}

\newtheorem{Theorem}{Theorem}[section]
\newtheorem*{Theorem*}{Theorem}
\newtheorem{Corollary}[Theorem]{Corollary}
\newtheorem{Lemma}[Theorem]{Lemma}
\newtheorem{Proposition}[Theorem]{Proposition}

\theoremstyle{definition}
\newtheorem{Definition}[Theorem]{Definition}

\newtheorem{Remark}[Theorem]{Remark}

\providecommand{\dd}{\mathop{}\!\mathrm{d}}
\providecommand{\dt}[2][]{\frac{\mathrm{d}^{#1}#2}{\mathrm{d}t^{#1}}}

\begin{document}

\allowdisplaybreaks

\newcommand{\arXivNumber}{2603.01132}

\renewcommand{\PaperNumber}{114}

\FirstPageHeading

\ShortArticleName{Krall-Type Orthogonal Polynomials and Integrable Isomonodromic Deformations}

\ArticleName{Krall-Type Orthogonal Polynomials \\ and Integrable Isomonodromic Deformations}

\Author{Luc HAINE}

\AuthorNameForHeading{L.~Haine}

\Address{UCLouvain, Institut de Recherche en Math\'ematique et Physique, IRMP,\\
Chemin du Cyclotron 2, 1348 Louvain-la-Neuve, Belgium}
\Email{\mail{luc.haine@uclouvain.be}}

\ArticleDates{Received March 03, 2026, in final form October 26, 2026; Published online November 14, 2026}

\Abstract{Krall-type polynomials are orthogonal polynomials for a~Stieltjes measure obtained by adding jumps at the boundary of the orthogonality interval of either the generalized Laguerre polynomials or the Jacobi polynomials. We show that both the recurrence relations and the second-order linear differential equations defining these polynomials are explicitly determined in terms of specific solutions of certain integrable systems. When only one jump is present, this leads to integrable cases of the Painlev\'e III or the Painlev\'e V equation. In the case of two jumps, first studied by Koornwinder~[\textit{Canad. Math. Bull.} \textbf{27} (1984), 205--214], we obtain a~new integrable system of partial differential equations of Schlesinger type. When the jumps are equal and the starting polynomials are the Gegenbauer polynomials, this system reduces to an integrable case of the Painlev\'e V equation.}


\Classification{33C45; 34M56}

\section{Introduction}
Let $w(x)$ denote the weight function for either the generalized Laguerre or Jacobi orthogonal polynomials. For the generalized Laguerre polynomials, the interval of orthogonality is $[a,b)=[0,\infty)$, while for the Jacobi polynomials we choose either $[a,b]=[0,1]$ or $[a,b]=[-1,1]$, for convenience. Krall-type polynomials are orthogonal polynomials with respect to the Stieltjes measure $\dd\sigma_{t_1,t_2}(x)$, defined as follows
\begin{gather}
\sigma_{t_1,t_2}(x)=\begin{cases} 0 & \mbox{if}\ x\leq a,\\
\displaystyle \frac{1}{t_1}+\int_a^x w(t) \dd t & \mbox{if}\ a<x<b,\\
\displaystyle\frac{1}{t_1}+\frac{1}{t_2}+\int_a^b w(t)\dd t & \mbox{if}\ x\geq b,\end{cases}
\label{kt}
\end{gather}
with $t_1,t_2 >0$, and possibly $t_1$ or $t_2$ infinite, corresponding to a~single jump. Alternatively, Krall-type polynomials are orthogonal with respect to the weight distribution
\begin{gather}
\sigma_{t_1,t_2}'(x)=w(x)(H(x-a)-H(x-b))+\frac{1}{t_1}\delta(x-a)+\frac{1}{t_2}\delta(x-b),\label{kd}
\end{gather}
where $H(x)$ is the Heaviside function and $\delta(x)$ is the Dirac delta function. Throughout, we normalize Krall-type orthogonal polynomials $p_n(x,t_1,t_2), n\geq 0$, to be monic, so they satisfy the three-term recurrence relation
\begin{gather}
xp_n(x,t_1,t_2)=p_{n+1}(x,t_1,t_2)+b_n(t_1,t_2)p_n(x,t_1,t_2)+a_n(t_1,t_2) p_{n-1}(x,t_1,t_2),
\label{re}
\end{gather}
for $n\geq 0$ with $p_{-1}=0$ and $p_0=1$. The terminology ``Krall-type polynomials'' is justified by the fact that certain special cases first appeared in the work of H.L.~Krall~\cite{K1, K2}, who studied the problem of finding all Chebyshev sets of polynomials that are eigenfunctions of a~differential operator of arbitrary order, generalizing the classical orthogonal polynomials (which correspond to order two). Except for the first non trivial case of order four, which was completely solved in~\cite{K2}, the general problem remains partially open. We refer the reader to A.M.~Krall~\cite{KAM1, KAM2} and Littlejohn~\cite{LLL1}, who revived the subject in the 1980's, for more references, and to Kwon and Lee~\cite{KL} for characterizations of Bochner--Krall orthogonal polynomials of Jacobi type. Using an explicit formula for Krall-type polynomials discovered by Koornwinder~\cite{Ko}, J.~Koekoek and R.~Koekoek \mbox{\cite{KK1,KK2}} showed that, in general, Krall-type polynomials solve an ``infinite order version'' of Krall's problem, reducing to finite order only for special choices of the classical weights~$w(x)$.\looseness=-1

In previous work~\cite{GH1}, it was shown that the polynomials discovered by Krall~\cite{K2} can be obtained via Darboux transformations, starting from certain instances of classical orthogonal polynomials. Further developments~\cite{GHH, GY, H, I2, I1} led to new solutions to Krall's problem as posed in~\cite{K1}. Krall-type polynomials, as defined in~\eqref{kt}, were also constructed by Zhedanov~\cite{Zh} using Darboux transformations. In contrast, the present work does not rely on Darboux transformations; instead, it constructs Krall-type polynomials explicitly, using a~method originating from a~remarkable work by Laguerre~\cite{La}. For a~historical review, see Dieudonn\'e~\cite{Di} as well as Magnus and Ronveaux~\cite{MR}. As noted by Nuttall~\cite{N}, G.~Chudnovsky observed that Laguerre's method allows orthogonal polynomials corresponding to certain generalized Jacobi weights with three factors to be constructed in terms of specific solutions of the Painlev\'e VI equation (see also~\cite{CC, M}). Similarly, in previous work~\cite{HS}, certain orthogonal polynomials arising in the context of the Jacobi polynomial ensemble were characterized in terms of the Painlev\'e VI equation, at first sight using unrelated tools from the theory of the Toda lattice hierarchy and its master symmetries~\cite{AVM, Da, FG}. As will be shown, these tools are closely tied to isomonodromic deformations. The central result of this work is that Krall-type polynomials lead to \emph{integrable} cases of the Painlev\'e equations and, more generally, to \emph{integrable} isomonodromic deformations. The \emph{integrability} of these equations yields explicit expressions for Krall-type polynomials. For a~comprehensive overview of the intriguing connections between Painlev\'e equations and orthogonal polynomials, up to 2018, we refer the reader to Van Assche's excellent monograph~\cite{V}. However, the \emph{integrable} cases of interest in this paper are frequently excluded from these studies.

To introduce our main approach, let us consider the Stieltjes transform
\begin{gather}
f_{t_1,t_2}(x)=\int_a^b \frac{\dd\sigma_{t_1,t_2}(s)}{x-s},\qquad x\in \mathbb{C}\setminus [a,b],\label{St}
\end{gather}
with $\sigma_{t_1,t_2}$ given in~\eqref{kt}. Throughout this paper, the prime symbol $'$ denotes $\mathrm{d}/\mathrm{d}x$, and $t_1$, $t_2$ are treated as parameters. We observe that Krall-type polynomials satisfy Laguerre's hypothesis, which states that
\begin{gather}
W(x)f'_{t_1,t_2}(x)=2V(x)f_{t_1,t_2}(x)+U(x,t_1,t_2),\label{scKt}
\end{gather}
where $U$, $V$ and $W$ are polynomials in $x$. The fact that only $U$ depends of $(t_1,t_2)$ is a~characteristic feature of Krall-type polynomials. As Laguerre~\cite{La} established, condition~\eqref{scKt} implies that for all $(t_1,t_2)$, the $x$-dependent functions $p_n(x,t_1,t_2)$ and \smash{$\frac{\varepsilon_n(x,t_1,t_2)}{v(x)}$}, with $\varepsilon_n(x,t_1,t_2)$ and~$v(x)$ defined by
\begin{gather}
\varepsilon_n(x,t_1,t_2)=\int_a^b\frac{p_n(s,t_1,t_2)}{x-s}\dd\sigma_{t_1,t_2}(s),\qquad x\in \mathbb{C}\setminus [a,b],\nonumber\\
v(x)=\exp\left(\int \frac{2V(x)}{W(x)}\dd x\right),\label{ep}
\end{gather}
form a~basis of solutions for a~linear second-order differential equation
\begin{gather}
W\Theta_n g_n''+[(2V+W')\Theta_n-W\Theta_n']g_n'+K_n g_n=0.\label{lae}
\end{gather}
Here, $\Theta_n$ and $K_n$ are polynomials in $x$ (depending on $n$, $t_1$, $t_2$) whose degrees are independent of $n$. In the limit as $t_1,t_2 \to \infty$, this equation reduces to the standard second-order differential equation satisfied by either generalized Laguerre or Jacobi polynomials. In this limit, $\Theta_n$ becomes independent of $n$, and $K_n$ becomes a~constant. As will be demonstrated in Section \ref{sec2}, the linear second-order differential equation~\eqref{lae} is isomonodromic with respect to the parameters~$t_1$ and~$t_2$. When there is only one jump or two equal jumps, this indicates that certain Painlev\'e equations may be lurking in the background.

Among the six Painlev\'e equations $P_{\rm I}$--$P_{\rm VI}$, the only ones which admit a~first integral which is an elementary function of $q$, $\dt{q}$ and $t$, are the $P_{\rm III}$ equation
\begin{gather}
(P_{\rm III})  \dt[2]{q}=\frac{1}{q}\left(\dt{q}\right)^2-\frac{1}{t}\dt{q}+\frac{1}{t}\bigl(aq^2+b\bigr)+cq^3+\frac{\rm d}{q},\label{P3}
\end{gather}
when $a=c=0$, with first integral
\begin{gather*}
\frac{t^2}{q^2}\left(\dt{q}\right)^2-\frac{2t}{q}\dt{q}+{\rm d}\frac{t^2}{q^2}+2b\frac{t}{q}=C,
\end{gather*}
or when $b=d=0$, with first integral
\begin{gather}
\frac{t^2}{q^2}\left(\dt{q}\right)^2+\frac{2t}{q}\dt{q}-2atq-ct^2q^2=C,\label{FIP32}
\end{gather}
and the $P_{\rm V}$ equation
\begin{gather}
(P_{\rm V}) \dt[2]{q}=\left(\frac{1}{2q}+\frac{1}{q-1}\right)\left(\dt{q}\right)^2-\frac{1}{t}\dt{q}+\frac{(q-1)^2}{t^2}\left(aq+\frac{b}{q}\right)+\frac{cq}{t}+\frac{{\rm d} q(q+1)}{q-1},\label{P5}
\end{gather}
when $c=d=0$, with first integral
\begin{gather}
\frac{t^2}{2q(q-1)^2}\left(\dt{q}\right)^2-aq+\frac{b}{q}=C,\label{FIP5}
\end{gather}
with $C$ an arbitrary constant. This result was established by Gromak~\cite{G2, G1}, building on work by Luka\v{s}evi\v{c}~\cite{Lu}, who discovered these cases and showed that they reduce to elementary quadratures. Subsequently, \.Zo\l\c{a}dek and Filipuk~\cite{ZF} offered a~new proof. It is worth mentioning that the Painlev\'e property for $P_{\rm III}$ and $P_{\rm V}$ was rigorously established by Hinkkanen and Laine~\cite{HL2, HL1} as part of a~general program for all Painlev\'e equations. While these references serve as comprehensive sources on the subject, they also contain an accessible discussion of the integrable cases of interest here. Alternatively, the reader may consult the authoritative treatise by Gromak, Laine, and Shimomura~\cite[Sections~29 and~36]{GLS}. Specifically, we shall establish the following theorem.
\begin{Theorem} \label{theorem krallpainleve} %\label{theorem 1.1}
The recurrence relation~\eqref{re} and the differential equation~\eqref{lae} for
\begin{itemize}\itemsep=0pt
\item[$(a)$] the Krall--Laguerre-type polynomials, with weight distribution
\begin{gather}
\frac{1}{\Gamma(\alpha+1)}x^{\alpha}{\rm e}^{-x}H(x)+\frac{1}{t}\delta (x),\qquad \alpha>-1,\label{wdKLa}
\end{gather}

\item[$(b)$] the Krall--Jacobi-type polynomials, with weight distribution
\begin{gather}
\frac{\Gamma(\alpha+\beta+2)}{\Gamma(\alpha+1)\Gamma(\beta+1)}x^{\beta}(1-x)^{\alpha}(H(x)-H(x-1))+\frac{1}{t} \delta(x),\qquad \alpha,\beta>-1,\label{wdKJ}
\end{gather}

\item[$(c)$] the Krall--Gegenbauer-type polynomials with weight distribution
\begin{gather}
\frac{\Gamma(\alpha+3/2)}{\sqrt{\pi} \Gamma(\alpha+1)}\bigl(1-x^2\bigr)^{\alpha}(H(x+1)-H(x-1))
+\frac{1}{t}(\delta(x+1)+\delta(x-1)),\nonumber\\
  \alpha>-1,\label{wdKG}
\end{gather}
are completely determined by a~sequence of solutions $q_n(t)$, $n\geq 1$, of integrable cases of the~$P_{\rm III}$ or the $P_{\rm V}$ equation. Specifically, for cases $(a)$, $(b)$ and $(c)$, these correspond to
\begin{itemize}\itemsep=2pt
\item[$(a)$] $P_{\rm III}$ with \smash{$a=-\frac{2n+1+\alpha}{(\alpha+1)^2}$}, $b=0$, \smash{$c=\frac{1}{(\alpha+1)^2}$}, $d=0$,
\item[$(b)$] $P_{\rm V}$ with \smash{$a=\frac{(\alpha+\beta+2n+1)^2}{2(\beta+1)^2}$}, \smash{$b=-\frac{\alpha^2}{2(\beta+1)^2}$}, $c=0$, $d=0$,
\item[$(c)$] $P_{\rm V}$ with \smash{$a=\frac{(2\alpha+2n+1)^2}{8(\alpha+1)^2}$}, \smash{$b=-\frac{1}{8(\alpha+1)^2}$}, $c=0$, $d=0$.
\end{itemize}
\end{itemize}
\end{Theorem}

For Koornwinder's polynomials studied in~\cite{Ko}, with the weight distribution
\begin{gather}
\frac{\Gamma(\alpha+\beta+2)}{2^{\alpha+\beta+1}\Gamma(\alpha+1)\Gamma(\beta+1)}(1-x)^{\alpha}(1+x)^{\beta}\bigl(H(x+1)-H(x-1)\bigr)\nonumber\\
\qquad+\frac{1}{t_1}\delta(x+1)+\frac{1}{t_2}\delta(x-1), \qquad \alpha, \beta>-1,\label{wdK}
\end{gather}
one obtains a~similar result, except that the polynomials $\Theta_n$ and $K_n$ in Laguerre's equation~\eqref{lae} now involve two functions depending on two variables, $t_1$ and $t_2$, whereas Theorem~\ref{theorem krallpainleve} involves only one function depending on a~single variable, $t$. These two functions satisfy a~Schlesinger system of partial differential equations, which we prove to be \emph{integrable} by exhibiting its first integrals and solving it via separation of variables. When $t_1=t_2=t$ and $\alpha=\beta$, the two functions become equal and the system reduces to an integrable case of the Painlev\'e V equation, as stated in Theorem~\ref{theorem krallpainleve}\,(c).

The paper is organized as follows. Section \ref{sec2} summarizes known results on semi-classical orthogonal polynomials, isomonodromic deformations and the master symmetries of the Toda lattice hierarchy that will be needed in the rest of the paper. Sections \ref{sec3} and \ref{sec4} establish, respectively, cases (a) and (b) of Theorem~\ref{theorem krallpainleve}. Section \ref{sec5} deals with Koornwinder's polynomials and introduces a~new integrable system of partial differential equations of Schlesinger type. Specializing this system to Krall--Gegenbauer-type polynomials with equal jumps establishes case~(c) of Theorem~\ref{theorem krallpainleve}. In Sections \ref{sec3} to \ref{sec5}, we also derive from the Schlesinger equations, differential and partial differential equations satisfied by the coefficients of the recurrence relations defining Krall-type polynomials, which can be expressed in terms of the Toda lattice hierarchy and its master symmetries. Some particular cases of these equations have already appeared in~\cite[Sections 7 and 8]{GH2}, inspired by the results of Zubelli and Magri~\cite{ZM}, which relate the work of Duistermaat and Gr\"unbaum~\cite{DG} on differential equations in the spectral parameter to the Korteweg--de Vries hierarchy and its master symmetries. However, in~\cite{GH2} the relation to isomonodromic deformations was not established. The potential relation of~\cite{DG,ZM} to integrable isomonodromic deformations remains an interesting open problem.

Some special cases of the second-order differential equations satisfied by Krall-type polynomials were obtained by Littlejohn and Shore~\cite{LLS}, and by Littlejohn~\cite{LLL2}. Their method requires these polynomials to be solutions of Krall's original problem~\cite{K1}, and the differential operator of which these polynomials are eigenfunctions must be explicitly known. Magnus, Ndayiragije, and Ronveaux~\cite{MNR} derived the general differential equations satisfied by the Krall--Laguerre-type and the Krall--Jacobi-type polynomials. Although the authors based their work on Laguerre's theory~\cite{La}, they did not observe the connection to integrable cases of the Painlev\'e III and the Painlev\'e V equation. Instead, they resorted to computing Krall-type polynomials via the Darboux transformation method~\cite{GH1, GHH, GY, Zh} to determine $\Theta_n$ and $K_n$ in~\eqref{lae} -- an approach we avoid in this paper. Koornwinder~\cite{Ko} devised a~method for deriving the second-order differential equation in the general case, without using Laguerre's equation~\eqref{lae}. Introducing different tools, Kiesel and Wimp~\cite{KW} derived this equation explicitly, although not in its simplest form. Finally, Section \ref{sec6} relates our results to Koornwinder's method.

\section[Semi-classical orthogonal polynomials and master symmetries of the Toda lattice hierarchy]{Semi-classical orthogonal polynomials and master symmetries\\ of the Toda lattice hierarchy}\label{sec2}
In this section, we briefly survey the theory of semi-classical orthogonal polynomials and their relation to isomonodromic deformations, focusing specifically on the Krall-type polynomials studied in this paper. For foundational references on semi-classical orthogonal polynomials, we refer the reader to Hendriksen and van Rossum~\cite{HV}, Maroni~\cite{Mar}, and Ronveaux~\cite{Ro}. The connection to Painlev\'e equations was remarkably developed by Magnus~\cite{M}. For an excellent recent overview of these contributions, see Rebocho~\cite{Re}. Adopting the convention of Laguerre~\cite{La}, we normalize the polynomials to be monic and preserve most of his notation. Furthermore, we review the Toda lattice hierarchy and its master symmetries, which are closely linked to isomonodromic deformations. We note that after completing this work, M.~Bertola kindly informed us that the connection between semi-classical orthogonal polynomials and isomonodromic deformations in the context of matrix models was independently investigated in~\cite{BEH}.

Let $\sigma(x)$ be a~non-decreasing function with infinitely many points of increase in the finite or infinite interval $[a,b]$, and assume that the moments
\begin{gather*}
\mu_k=\int_a^b x^k\dd\sigma(x),\qquad  k\geq 0,
\end{gather*}
exist. Let $p_n(x)$, $n\geq 0$, be the system of orthogonal polynomials with respect to the Stieltjes measure $\dd\sigma(x)$, normalized to be monic, satisfying the orthogonality relations
\begin{gather}
\int_a^b p_m(x) p_n(x)\dd\sigma(x)=h_n\delta_{mn},\label{h}
\end{gather}
where $\delta_{mn}$ is the Kronecker symbol. The orthogonal polynomials (normalized to be monic) have a~well known expression in terms of the moments
\begin{gather}
p_n(x)=\begin{vmatrix}\mu_0&\mu_1&\ldots&\mu_{n}\\
\mu_1&\mu_2&\ldots&\mu_{n+1}\\
\vdots&\vdots&\ddots&\vdots\\\
\mu_{n-1}&\mu_n&\ldots &\mu_{2n-1}\\
1& x&\ldots &x^n\end{vmatrix}/\Delta_{n-1},\qquad n=1,2,\ldots,
\label{opm}
\end{gather}
with
$
\Delta_{n-1}=\det(\mu_{i+j-2})_{1\leq i,j\leq n}$.

The Stieltjes transform $f(x)$ of $\dd\sigma(x)$ generally admits the asymptotic expansion
\begin{gather}
f(x)=\int_a^b \frac{\dd\sigma(s)}{x-s}\sim\sum_{k=0}^\infty \frac{\mu_k}{x^{k+1}},\qquad \mbox{as}\quad  x\to\infty,\label{Stm}
\end{gather}
which converges for $\vert x\vert>\max(\vert a\vert, \vert b\vert)$ if the interval $[a,b]$ is compact. The motivation behind Laguerre's work~\cite{La}, which predates and influenced Stieltjes (see the introduction of~\cite{S}), was to explicitly approximate~\eqref{Stm} (in some special cases, most notably what is now called the generalized Laguerre weight) by a~convergent continued fraction. The starting point is the identity
\begin{align}
p_n(x)f(x)&=\int_a^b\frac{p_n(x)-p_n(s)}{x-s}\dd\sigma(s)+\int_a^b\frac{p_n(s)}{x-s}\dd\sigma(s),\nonumber\\
&=q_n(x)+\varepsilon_n(x),\qquad  x\in\mathbb{C}\setminus [a,b],\label{Stp}
\end{align}
which defines $q_n(x)$, the so-called numerator polynomials, and the functions $\varepsilon_n(x)$. The orthogonal polynomials
\begin{gather}
p_n(x)=x^n+c_nx^{n-1}+d_nx^{n-2}+e_nx^{n-3}+ \cdots,\label{cp}
\end{gather}
satisfy the three-term recurrence relation
\begin{gather}
p_{n+1}(x)=(x-b_n)p_n(x)-a_n p_{n-1}(x), \qquad n\geq 0,\label{recg}
\end{gather}
subject to the initial conditions $p_{-1}=0$ and $p_0=1$, where the coefficients are given by
\begin{gather}
a_n=\frac{h_n}{h_{n-1}},\qquad  n\geq 1, \qquad b_n=c_n-c_{n+1},\qquad n\geq 0,\label{abch}
\end{gather}
and $h_n$ is defined in~\eqref{h}. Similarly, the polynomials $q_n(x)$ satisfy~\eqref{recg} for $n\geq 1$, with the initial conditions $q_0=0$ and $q_1=\mu_0$. From~\eqref{Stp}, using~\eqref{h}, it follows that
\begin{align}
&\varepsilon_n=
h_n\left(\frac{1}{x^{n+1}}-\frac{c_{n+1}}{x^{n+2}}+\frac{c_{n+1}c_{n+2}-d_{n+2}}{x^{n+3}}\right.\nonumber\\
&\left.\hphantom{\varepsilon_n=}{} +\frac{c_{n+3}(d_{n+2}-c_{n+1}c_{n+2})+c_{n+1}d_{n+3}-e_{n+3}}{x^{n+4}}+\cdots\right),\qquad \mbox{as}\quad x\to \infty.\label{cep}
\end{align}
One computes that
\begin{align}
c_n&=-\sum_{i=0}^{n-1} b_i,\label{c}\\
d_n&=-\sum_{i=1}^{n-1}a_i+\sum_{0\leq i<j\leq n-1} b_ib_j,\label{d}\\
e_n&=\sum_{i=1}^{n-1}a_i\Bigg(\sum_{j=0}^{i-2}b_j+\sum_{j=i+1}^{n-1}b_j\Bigg)-\sum_{0\leq i<j<k\leq n-1} b_ib_jb_k.\label{e}
\end{align}

Formula~\eqref{opm} is often impracticable for explicit computation. Laguerre~\cite{La} introduced a~condition which sometimes leads to an explicit determination of the orthogonal polynomials.
\begin{Definition} A~family of orthogonal polynomials with respect to a~Stieltjes measure $\dd\sigma(x)$ is called \emph{semi-classical} if there exist polynomials $U(x)$, $V(x)$ and $W(x)$ such that
\begin{gather}
Wf'=2Vf+U,\label{sc}
\end{gather}
where $f$ is the Stieltjes transform~\eqref{Stm}. Necessarily, $\deg (U)\leq \max (\deg(W)-2,\deg(V)-1)$, where $\mbox{deg}(U)$, $\mbox{deg}(V)$ and $\mbox{deg} (W)$ denote the degrees of $U$, $V$ and $W$.
\end{Definition}
Under this condition, Laguerre~\cite{La} showed that $p_n(x)$ and $\varepsilon_n(x)/v(x)$, with $\varepsilon_n(x)$ defined in~\eqref{Stp} and \smash{$v(x)=\exp\bigl(\int \frac{2V(x)}{W(x)}\dd x\bigr)$}, form a~basis of solutions for a~linear second-order differential equation, already displayed in~\eqref{lae}. The polynomials $\Theta_n$ and $K_n$ in~\eqref{lae} are defined as follows
\begin{align}
\Theta_n&=\frac{W(p_nq'_n-q_np'_n)-2Vq_np_n-Up_n^2}{h_n},\label{theta}\\
K_n&=(V-\Omega_n)'\Theta_n-(V-\Omega_n)\Theta_n'+\frac{\bigl(a_n\Theta_{n-1}\Theta_n+V^2-\Omega_n^2\bigr)\Theta_n}{W},\label{K}
\end{align}
where $q_n$ is the numerator polynomial defined in~\eqref{Stp}, and
\begin{gather}
\Omega_n=\frac{W(q'_np_{n-1}-p'_nq_{n-1})-V(p_nq_{n-1}+q_np_{n-1})-Up_n p_{n-1}}{h_{n-1}}.\label{omega}
\end{gather}
The fact that $K_n$ is a~polynomial follows from establishing that
\begin{gather}
a_n\Theta_{n-1}\Theta_n+V^2-\Omega_n^2=W\sum_{i=0}^{n-1} \Theta_i,\label{Kdiv}
\end{gather}
which implies that the left-hand side is divisible by $W$. Using~\eqref{Stp} and~\eqref{sc}, one can verify that the two polynomials $\Theta_n$ and $\Omega_n$ can be written as
\begin{align}
\Theta_n&=\frac{W(\varepsilon_n p'_n-p_n \varepsilon'_n)+2V\varepsilon_n p_n}{h_n},\label{ctheta}\\
\Omega_n&=\frac{W(\varepsilon_{n-1}p'_n-\varepsilon'_np_{n-1})+V(\varepsilon_{n-1}p_n+\varepsilon_np_{n-1})}{h_{n-1}},\label{comega}
\end{align}
which shows that the degree of these polynomials (which are $n$ dependent) is bounded by
\begin{gather*}
\deg(\Theta_n)\leq \max(\deg (W)-2,\deg(V)-1), \qquad \deg(\Omega_n)\leq \max (\deg(W)-1, \deg(V)),
\end{gather*}
hence $K_n$ is also of bounded degree. Since $\Theta_n$ and $\Omega_n$ are polynomials, formulas~\eqref{ctheta} and~\eqref{comega} allow their computation in terms of $c_n, d_n, e_n,\ldots $ from~\eqref{cp} and~\eqref{cep}.

For Krall-type polynomials as defined in~\eqref{kd}, the Stieltjes transform~\eqref{St} is given by
\begin{gather}
f_{t_1,t_2}(x)=\int_{a}^{b}\frac{w(s)}{x-s}\dd s +\frac{1}{t_1(x-a)}+\frac{1}{t_2(x-b)}.\label{SKt}
\end{gather}
Hence, applying the Stieltjes inversion formula (also called the Sokhotski--Plemelj formula) and denoting $f^{\pm}_{t_1,t_2}(x)=\lim_{\varepsilon \to 0^+} f_{t_1,t_2}(x\pm {\rm i}\varepsilon)$, $x\in(a,b)$, it follows that
\begin{gather}
f^{-}_{t_1,t_2}(x)-f^{+}_{t_1,t_2}(x)=2\pi {\rm i} w(x), \qquad \forall x\in(a,b).\label{SPi}
\end{gather}
Using~\eqref{scKt}, we deduce from~\eqref{SPi} that
$
W(x)w'(x)=2V(x)w(x)$, $ \forall x\in (a,b)$.
Therefore, \smash{$v(x)=\exp\bigl(\int \frac{2V(x)}{W(x)}\dd x\bigr)$} is given by
\begin{gather}
v(x)=cx^{\alpha} {\rm e}^{-x},\qquad cx^{\beta} (x-1)^{\alpha}, \qquad\mbox{and}\qquad c(x-1)^{\alpha}(x+1)^{\beta},\label{vex}
\end{gather}
in the Krall--Laguerre, Krall--Jacobi and Koornwinder cases respectively, where $c$ is a~nonzero constant. The second solution to the Laguerre equation~\eqref{lae}, $\varepsilon_n(x,t_1,t_2)/v(x)$ with $\varepsilon_n(x,t_1,t_2)$ as in~\eqref{ep}, is uniquely determined by choosing an appropriate branch of $v(x)$, which can always be taken to be holomorphic on $\mathbb{C}\setminus [a,\infty)$. When $p_n(x,t_1,t_2)$ and $\varepsilon_n(x,t_1,t_2)/v(x)$ are analytically continued around positively oriented closed loops encircling $a$ or $b$, the polynomial $p_n(x,t_1,t_2)$ returns to its original value, while the second solution maps to a~linear combination of the initial solutions. This behavior determines the $2\times2$ monodromy matrices $M_a(t_1,t_2)$ and $M_b(t_1,t_2)$, depending a~priori on $t_1$ and $t_2$, and defined by
\begin{gather*}
(p_n,\varepsilon_n/v)\mapsto (p_n,\varepsilon_n/v)M_{a}(t_1,t_2), \qquad (p_n,\varepsilon_n/v)\mapsto (p_n,\varepsilon_n/v)M_b(t_1,t_2).
\end{gather*}
\begin{Proposition} \label{proposition isomonodromy} %\label{proposition 2.1}
The monodromy matrices $M_a$ and $M_b$, computed around two positively oriented closed loops with a~base point in the upper half-plane encircling $a$ and $b$ respectively, are given by
\begin{gather}
M_0=\begin{pmatrix} 1& 2\pi {\rm i} {\rm e}^{-2\pi {\rm i} \alpha}\\ 0& {\rm e}^{-2\pi {\rm i} \alpha}\end{pmatrix},\label{mola}
\end{gather}
in the Krall--Laguerre case $($where $a=0$ and $b=\infty${\rm)}, and
\begin{gather}
M_a=\begin{pmatrix} 1& 2\pi {\rm i} {\rm e}^{-\pi {\rm i} (\alpha+2\beta)}\\ 0&{\rm e}^{-2\pi {\rm i} \beta}\end{pmatrix}, \qquad M_b=\begin{pmatrix} 1&-2\pi {\rm i} {\rm e}^{-\pi {\rm i} \alpha}\\ 0& {\rm e}^{-2\pi {\rm i} \alpha}\end{pmatrix},\label{moja}
\end{gather}
in the Krall--Jacobi $($with $a=0$, $b=1${\rm)} and the Koornwinder $($with $a=-1$, $b=1${\rm)} cases. In particular, they are independent of $(t_1,t_2)$.
\end{Proposition}
\begin{proof}
Define $v^{\pm}(x)=\lim_{\varepsilon \to 0^+} v(x\pm {\rm i}\varepsilon), x\in (a,\infty)$, $\varepsilon^{\pm}_n(x,t_1,t_2)=\lim_{\varepsilon \to 0^+} \varepsilon_n (x\pm {\rm i}\varepsilon,t_1,t_2)$, $x\in(a,b)$. From~\eqref{Stp} and~\eqref{SPi}, we find that
\begin{gather*}
\left(\frac{\varepsilon_n}{v}\right)^{-}(x,t_1,t_2)=2\pi {\rm i}\frac{w(x)}{v^{-}(x)}p_n(x,t_1,t_2)+\frac{v^{+}(x)}{v^{-}(x)}\left(\frac{\varepsilon_n}{v}\right)^{+}(x,t_1,t_2), \qquad x\in(a,b).
\end{gather*}
This shows that the monodromy matrix around $a$ is given by
\begin{gather}
M_a=\begin{pmatrix} 1&\dfrac{2\pi {\rm i} w(x)}{v^{-}(x)}\vspace{1mm}\\0&\dfrac{v^{+}(x)}{v^{-}(x)}\end{pmatrix},\qquad x\in(a,b),\label{Ma}
\end{gather}
since a~positively oriented closed loop around $a$ crosses the cut once on $(a,b)$. Similarly, when $b$ is finite, we find that the monodromy matrix around $b$ is given by
\begin{gather}
M_b=\begin{pmatrix} 1&-\dfrac{2\pi {\rm i} w(x)}{v^+(x)}\vspace{1mm}\\0&\dfrac{v^{-}(x)v^+(y)}{v^+(x)v^{-}(y)}\end{pmatrix},\qquad  x\in(a,b),\qquad   y\in(b,\infty),\label{Mb}
\end{gather}
taking into account that a~positively oriented closed loop around $b$ now crosses the cut twice: once on $(a,b)$ and once on $(b,\infty)$. Choosing $c>0$ in~\eqref{vex} such that $\big\vert v^{\pm}(x)\big\vert=w(x)$ for $x\in (a,b)$, and specializing~\eqref{Ma} and~\eqref{Mb} to the three cases, yields~\eqref{mola} and~\eqref{moja}.
\end{proof}

Setting
\begin{gather}
Z=\begin{pmatrix} p_n&\dfrac{\varepsilon_n}{v}\vspace{1mm}\\p_{n-1}&\dfrac{\varepsilon_{n-1}}{v}\end{pmatrix},\label{Z}
\end{gather}
recalling that $Z$ depends on $(x, t_1,t_2)$, Laguerre's equation~\eqref{lae} can be written as a~linear system of two first-order equations
\begin{gather}
\frac{\partial Z}{\partial x}= A Z,\label{lmde}
\end{gather}
with
\begin{gather}
A=\frac{1}{W}\begin{pmatrix} \Omega_n-V&-a_n\Theta_n\\\Theta_{n-1}&-\Omega_n-V\end{pmatrix}.\label{lm}
\end{gather}
The matrix form~\eqref{lmde} of~\eqref{lae} appears in~\cite{M}. The matrix~\eqref{Z} -- without dividing the second column by $v$ -- was employed by Fokas, Its, and Kitaev~\cite{FIK} and Deift~\cite{De} to characterize orthogonal polynomials as solutions of a~Riemann--Hilbert problem.

By Proposition~\ref{proposition isomonodromy}, the monodromy of the solutions to~\eqref{lmde} around $x=a$ and $x=b$ is independent of $(t_1, t_2)$ for Krall-type polynomials; hence, the matrices
\begin{gather}
H_i=\frac{\partial Z}{\partial t_i} Z^{-1}, \qquad i=1, 2,\label{defH}
\end{gather}
are single-valued functions of $x$ on $\mathbb{C}\setminus \{a,b\}$. From~\eqref{lmde} and~\eqref{defH}, a~straightforward computation yields
\begin{align*}
\frac{\partial^2 Z}{\partial x \partial t_i}&= \frac{\partial}{\partial x} \frac{\partial Z}{\partial t_i}=\frac{\partial H_i}{\partial x}Z+H_i\frac{\partial Z}{\partial x}=\frac{\partial H_i}{\partial x}Z+H_iAZ\\
&=\frac{\partial}{\partial t_i}\frac{\partial Z}{\partial x}=\frac{\partial A}{\partial t_i}Z+A\frac{\partial Z}{\partial t_i}=\frac{\partial A}{\partial t_i}Z+AH_iZ,
\end{align*}
which leads to the system of partial differential equations
\begin{gather}
\frac{\partial A}{\partial t_i}=\frac{\partial H_i}{\partial x}+[H_i,A],\qquad  i=1, 2,\label{PSL}
\end{gather}
known as the \emph{Schlesinger equations}.

Since
\smash{$
\det Z=\frac{p_n\varepsilon_{n-1}-p_{n-1}\varepsilon_n}{v}=\frac{h_{n-1}}{v}$},
$h_n$ defined as in~\eqref{h}, recalling that \smash{$\frac{\partial v}{\partial t_1}=\frac{\partial v}{\partial t_2}=0$}, from~\eqref{Z} and~\eqref{defH} we obtain
\begin{gather}
H_i=\frac{1}{h_{n-1}}
\begin{pmatrix} \dfrac{\partial p_n}{\partial t_i}\varepsilon_{n-1}-p_{n-1}\dfrac{\partial \varepsilon_n}{\partial t_i}& p_n\dfrac{\partial \varepsilon_n}{\partial t_i}-\dfrac{\partial p_n}{\partial t_i}\varepsilon_n\vspace{1mm}\\
\dfrac{\partial p_{n-1}}{\partial t_i}\varepsilon_{n-1}-p_{n-1}\dfrac{\partial \varepsilon_{n-1}}{\partial t_i}& p_n\dfrac{\partial\varepsilon_{n-1}}{\partial t_i}-\dfrac{\partial p_{n-1}}{\partial t_i}\varepsilon_n\end{pmatrix},\qquad  i=1, 2.\label{Hi}
\end{gather}
From~\eqref{SKt}, we have
\begin{gather*}
\frac{\partial}{\partial t_1} f_{t_1,t_2}(x)=-\frac{1}{t_1^2(x-a)},\qquad \frac{\partial}{\partial t_2} f_{t_1,t_2}(x)=-\frac{1}{t_2^2(x-b)},
\end{gather*}
hence, equations~\eqref{Stp} and~\eqref{Hi} yield
\begin{align}
H_1={}&\frac{1}{h_{n-1}}\begin{pmatrix} p_{n-1}\dfrac{\partial q_n}{\partial t_1}-\dfrac{\partial p_n}{\partial t_1}q_{n-1}&\dfrac{\partial p_n}{\partial t_1}q_n-p_n\dfrac{\partial q_n}{\partial t_1}\vspace{1mm}\\
p_{n-1}\dfrac{\partial q_{n-1}}{\partial t_1}-\dfrac{\partial p_{n-1}}{\partial t_1}q_{n-1}&\dfrac{\partial p_{n-1}}{\partial t_1}q_n-p_n\dfrac{\partial q_{n-1}}{\partial t_1}\end{pmatrix}\nonumber\\
&+\frac{1}{h_{n-1}t_1^2(x-a)} \begin{pmatrix}p_{n-1}p_n&-p_n^2\\
p_{n-1}^2&-p_{n-1}p_n\end{pmatrix}.\label{H1}
\end{align}
Formula~\eqref{H1} shows that $H_1$, as a~function of $x$, is a~Laurent polynomial with a~simple pole at $x=a$. Returning to expression~\eqref{Hi}, and using~\eqref{cp} and~\eqref{cep}, one readily deduces that~${\lim_{x \to \infty} H_1=H_{\infty}}$, with
\begin{gather}
H_\infty=\frac{1}{h_{n-1}}\begin{pmatrix} 0&0\\0&\dfrac{\partial h_{n-1}}{\partial t_1}\end{pmatrix},\label{Hinf}
\end{gather}
whence
\begin{gather}
H_1=\frac{R_a}{x-a}+H_\infty,\label{cH1}
\end{gather}
with $R_a$ the residue of $H_1$ at $x=a$,
\begin{gather}
R_a=\mbox{Res}_{x=a} H_1=\frac{1}{ h_{n-1}t_1^2}\begin{pmatrix}p_{n-1}(a)p_n(a)&-p_n^2(a)\\
p_{n-1}^2(a)&-p_{n-1}(a)p_n(a)\end{pmatrix}.\label{Ra}
\end{gather}
The entries of $R_a$ are easily computed using~\eqref{theta} and~\eqref{omega}. As shown later for Krall-type polynomials, we always have $V(a)=W(a)=0$, which yields
\begin{gather}
p_n^2(a)=-\frac{h_n\Theta_n(a)}{U(a)}, \qquad p_{n-1}(a)p_n(a)=-\frac{h_{n-1}\Omega_n(a)}{U(a)}.\label{cRa}
\end{gather}
Naturally, a~similar formula holds for $H_2$, with $t_1$ replaced by $t_2$ and $a$ replaced by $b$.

For a~single jump -- meaning either the interval is $[a,\infty)$ with $t_1=t$, or the interval $[a,b]$ is finite with $t_1=t$ and $t_2=\infty$ -- the Schlesinger system \eqref {PSL} reduces to
\begin{gather}
\frac{\partial A}{\partial t}=\frac{\partial H}{\partial x}+[H,A],\label{OSL}
\end{gather}
where $A$ is given as in~\eqref{lm},
\begin{gather}
H=\frac{R_a}{x-a}+H_\infty, \qquad H_\infty=\frac{1}{h_{n-1}}\begin{pmatrix} 0&0\\0&\dfrac{\dd h_{n-1}}{\dd t}\end{pmatrix},\label{H}
\end{gather}
and $R_a$ is defined as in~\eqref{Ra}.

We conclude this section with some reminders about the Toda lattice hierarchy and its master symmetries. Let $L$ denote the semi-infinite Jacobi matrix
\begin{gather}
L=\begin{pmatrix}b_0&1&0&0&\ldots\\
a_1&b_1&1&0&\ldots\\
0&a_2&b_2&1&\ldots\\
\vdots&\vdots&\vdots&\ddots
\end{pmatrix},\label{todajac}
\end{gather}
and let $X(L)$ be the Lie derivative of $L$ in the direction of some vector field $X$. The Toda lattice hierarchy is the family of commuting vector fields $T_k, k \geq 0$, defined by
\begin{gather*}
T_k(L)= \big[L,\bigl(L^k\bigr)_{--}\big],\qquad  k\geq 1,
\end{gather*}
where $\bigl(L^k\bigr)_{--}$ denotes the strictly lower part of $L^k$. Setting $a_0=0$ and $b_{-1}=0$, the first two Toda vector fields are explicitly given by
\begin{gather}
T_1\colon\ \begin{cases} T_1(a_n)=a_n(b_n-b_{n-1}),\\T_1(b_n)=a_{n+1}-a_n,\end{cases}\label{T1}
\\
T_2\colon\ \begin{cases} T_2(a_n)=a_n\bigl(a_{n+1}-a_{n-1}+b_n^2-b_{n-1}^2\bigr),\\T_2(b_n)=a_{n+1}(b_{n+1}+b_n)-a_n(b_n+b_{n-1}).\end{cases}\label{T2}
\end{gather}

It is known that the Toda lattice hierarchy admits a~family of master symmetries $V_k$, $k\geq -1$, that satisfy the commutation relations
$
[V_k,V_l]=(l-k)V_{k+l}$, $ [V_k,T_l]=(l+1)T_{k+l}$,
thus forming a~semi-infinite Virasoro algebra. These master symmetries were studied by Damianou~\cite{Da} and by Adler and van Moerbeke~\cite{AVM} for symmetric Jacobi matrices. By defining the matrices $P$ and~$P^*$ via
\begin{gather*}
\frac{\dd p}{\dd x}=Pp, p=(p_0,p_1, p_2, \ldots)^{\mathsf{T}}, \qquad\frac{\dd p^*}{\dd x}=\bigl(P^*\bigr)^{\mathsf{T}}p^*,\qquad p^*=\left(\frac{p_0}{h_0}, \frac{p_1}{h_1},\frac{p_2}{h_2},\ldots\right)^{\mathsf{T}},
\end{gather*}
where $A^{\mathsf{T}}$ denotes the transpose of $A$, Faybusovich and Gekhtman~\cite{FG} found a~version adapted to the non-symmetric form of the Jacobi matrix~\eqref{todajac}. Specifically, the master symmetries acting on $L$ are given by
\begin{gather*}
V_k(L)=L^{k+1}+\big[\bigl(PL^{k+1}\bigr)_{+}-\bigl(L^{k+1}P^*\bigr)_{--},L\big],\qquad  k=-1,0,1,2,\ldots,
\end{gather*}
where $_{+}$ and $_{--}$ denote respectively the upper part and the strictly lower part of the corresponding matrices. The first few master symmetries are given by
\begin{gather}
V_{-1}\colon\ \begin{cases} V_{-1}(a_n)=0,\\
V_{-1}(b_n)=1,\end{cases}\label{Vm1}
\\
V_0\colon\ \begin{cases}V_0(a_n) =2a_n,\\
V_0(b_n)=b_n, \end{cases}\label{V0}
\\
V_1\colon\ \begin{cases} V_1(a_n)&=2a_n\bigl(nb_n-(n-2)b_{n-1}\bigr),\\
V_1(b_n)&=b_n^2+(2n+1)a_{n+1}-(2n-3)a_n,\end{cases}\label{V1}
\\
V_2\colon\ \begin{cases}
V_2(a_n)=2a_n\Bigl((3-n)\bigl(a_{n-1}+b_{n-1}^2\bigr)+a_n+n\bigl(a_{n+1}+b_n^2\bigr)\\
\phantom{V_2(a_n)=}{}+(b_n-b_{n-1})\sum\limits_{i=0}^{n-1}b_i\Bigr),\\
V_2(b_n)=b_n\bigl(b_n^2-2(n-2)a_n+2(n+1)a_{n+1}\bigr)+(2n+1)a_{n+1}b_{n+1}
\\
\phantom{V_2(b_n)=}{}+(5-2n)a_nb_{n-1}+2(a_{n+1}-a_n)\sum\limits_{i=0}^{n-1}b_i.\end{cases}\label{V2}
\end{gather}

\section[Krall-Laguerre-type polynomials and P\_III]{Krall--Laguerre-type polynomials and $\boldsymbol{P_{\rm III}}$}\label{sec3}
For convenience, we introduce the shifted factorial $(a)_n$ (Pochhammer's symbol) for a~real or complex number $a$ and any nonnegative integer $n$, where
$
(a)_0=1$ and $ (a)_n=a(a+1)\cdots (a+n-1)$, for $n>0$.
This section proves case (a) of Theorem~\ref{theorem krallpainleve}. Throughout this section, the prime symbol $'$ denotes~$\mathrm{d}/\mathrm{d}x$ for functions of $x$, which may depend on a~parameter $t$. More precisely, we will establish the following result.
\begin{Theorem} \label{theorem laguerrepainleve}
The Krall--Laguerre-type polynomials with weight distribution defined in~\eqref{wdKLa} are completely characterized by a~sequence of rational functions $y_n(t)=tq_n(t)$, $n\geq 0$, where~$q_n(t)$ is a~solution of the integrable case of the $P_{\rm III}$ equation~\eqref{P3}, with parameters
\begin{gather}
a=-\frac{2n+1+\alpha}{(\alpha+1)^2},\qquad b=0,\qquad c=\frac{1}{(\alpha+1)^2},\qquad d=0,\label{KLP3}
\end{gather}
and is uniquely determined by the asymptotic behavior
\begin{gather}
q_n(t)=\frac{(\alpha+1)(\alpha+1)_n}{n!t^2}+O\left(\frac{1}{t^3}\right),\qquad \mbox{as}\quad t\to \infty.\label{aKLP3}
\end{gather}
The coefficients of the recurrence relation~\eqref{re} satisfied by these polynomials are given by
\begin{align}
a_n&=-\frac{u_n(u_n+(2n+\alpha)y_n)}{y_n^2},\label{KLrea}\\
b_n&=2n+1+\alpha-y_n,\label{KLreb}
\end{align}
where
\begin{gather}
2u_n=y_n ^2-(2n+1+\alpha)y_n-(\alpha+1)t\dt{y_n}.\label{KLuy}
\end{gather}
The Laguerre differential equation~\eqref{lae} satisfied by these polynomials is given by
\begin{gather}
x(-x+y_n)g''_n+\bigl((x-y_n)(x-\alpha-2)+x\bigr)g'_n\nonumber\\
\qquad{}+\frac{1}{2}\left\{(\alpha+1)\left(\frac{t}{y_n}\dt{y_n}+1\right)+(2n-1)y_n-2nx\right\}g_n=0.\label{KLde}
\end{gather}
As $t\to\infty$, the coefficients~\eqref{KLrea} and~\eqref{KLreb} reduce to those of the recurrence relation satisfied by the generalized Laguerre polynomials, while the differential equation~\eqref{KLde} reduces to the standard differential equation for these polynomials
\begin{gather}
-xg''_n+(x-\alpha-1)g'_n-ng_n=0.\label{Lde}
\end{gather}
\end{Theorem}
\begin{Remark}
When \(\alpha = 0\), Littlejohn and Shore~\cite{LLS} obtained the differential equation~\eqref{KLde} satisfied by the Krall--Laguerre-type polynomials. They derived this equation from the fourth-order differential operator, discovered by Krall~\cite{K2}, for which these polynomials are eigenfunctions. The differential equation~\eqref{Lde} for the generalized Laguerre polynomials (occasionally called Sonine polynomials) appears in Laguerre’s papers~\cite{Lag, La} -- specifically in 1880 for $\alpha=0$ and in~1885 for $\alpha > -1$ -- as a~primary example of his theory. Equation~\eqref{Lde} also appears in an 1880 paper by Sonine~\cite{So} on cylindrical functions. In this notable, independent work, Sonine proved the orthogonality of these polynomials with respect to the weight function $x^{\alpha }{\rm e}^{-x}$ and provided their generating function. For a~detailed historical account, see~\cite[Chapter 31, p.~419]{BRZ}.
\end{Remark}
Theorem~\ref{theorem laguerrepainleve} follows from Lemmas~\ref{lemma laguerrescop} and~\ref{lemma laguerresode}, Lemma~\ref{lemma laguerretaylor} and Proposition~\ref{proposition laguerrefirstintegral}, which together yield the explicit formula for $y_n(t)$ presented in Proposition~\ref{proposition laguerresolution}.
\begin{Lemma} \label{lemma laguerrescop}
The Krall--Laguerre-type polynomials are semi-classical; precisely, the Stieltjes transform $f$ \eqref{Stm} satisfies the equation
\begin{gather}
Wf'=2Vf+U,\label{scKLa}
\end{gather}
with
\begin{gather}
W=x^2,\qquad 2V=-x^2+\alpha x,\qquad U=x+\frac{x-1-\alpha}{t}.\label{UVWklt}
\end{gather}
\end{Lemma}
\begin{proof}
For the Krall--Laguerre weight distribution~\eqref{wdKLa}, the Stieltjes transform~\eqref{St} is
\begin{gather*}
f_t(x)=F(x)+\frac{1}{tx},
\end{gather*}
with
\begin{gather*}
F(x)=\frac{1}{\Gamma(\alpha+1)}\int_0^\infty \frac{s^{\alpha}{\rm e}^{-s}}{x-s}\dd s.
\end{gather*}
One easily proves, as already observed by Laguerre~\cite{La}, that
$
xF'(x)=(\alpha-x)F(x)+1$,
from which~\eqref{scKLa} follows immediately, with $U$, $V$ and $W$ as in~\eqref{UVWklt}.
\end{proof}

From~\eqref{ctheta} and~\eqref{comega}, since $\Theta_n$ and $\Omega_n$ are polynomials, using~\eqref{cp} and~\eqref{cep}, one easily computes that
\begin{align}
\Theta_n&=-x+y_n,\label{KLtheta}\\
\Omega_n&=-\frac{x^2}{2}+\left(n+\frac{\alpha}{2}\right)x+u_n,\label{KLomega}
\end{align}
with
\begin{gather}
y_n=c_{n+1}-c_n+2n+1+\alpha=-b_n+2n+1+\alpha,\label{KLy}\\
2u_n=c_n^2-2c_n-c_nc_{n+1}+d_{n+1}-d_n-\frac{h_n}{h_{n-1}}=-2a_n+2\sum_{i=0}^{n-1}b_i,\label{KLu}
\end{gather}
where we have used~\eqref{abch}, \eqref{c} and~\eqref{d} to obtain $y_n$, $u_n$ in terms of $a_n$, $b_n$. From~\eqref{UVWklt}, \eqref{KLtheta} and~\eqref{KLomega}, the system of first order equations~\eqref{lmde} is explicitly given by
\begin{gather*}
Z'=\left(\frac{A_0}{x^2}+\frac{A_1}{x}+A_\infty\right)Z,
\end{gather*}
with
\begin{gather}
A_0=\begin{pmatrix} u_n&-a_ny_n\\y_{n-1}&-u_n\end{pmatrix},\qquad
A_1=\begin{pmatrix} n& a_n\\-1& -n-\alpha\end{pmatrix},\qquad
A_\infty=\begin{pmatrix}0&0\\0&1\end{pmatrix}.\label{KLA}
\end{gather}
From~\eqref{H} with $a=0$, we obtain
\begin{gather}
H=\frac{R_0}{x}+H_\infty, \qquad H_\infty=\frac{1}{h_{n-1}}\begin{pmatrix} 0&0\\0&\dfrac{{\rm d}h_{n-1}}{{\rm d}t}\end{pmatrix}, \label{KLH}
\end{gather}
and, using~\eqref{Ra} (with $a$=0 and $t_1=t$), \eqref{cRa}, \eqref{UVWklt}, \eqref{KLtheta} and~\eqref{KLomega}, we find
\begin{gather}
R_0=\frac{1}{(\alpha+1)t}\begin{pmatrix}u_n&-a_ny_n\\
y_{n-1}&-u_n\end{pmatrix}=\frac{A_0}{(\alpha+1)t}.\label{KLH0}
\end{gather}

Expressing that the polynomial in the left-hand side of~\eqref{Kdiv} is divisible by $W=x^2$, gives the two conditions
\begin{gather}
a_ny_{n-1}y_n-u_n^2=0,\label{dxklt}\\
a_n(y_{n-1}+y_n)+(2n+\alpha)u_n=0.\label{dx2klt}
\end{gather}
From these equations, we obtain
\begin{gather}
a_n=-\frac{u_n(u_n+(2n+\alpha)y_n)}{y_n^2},\label{KLauy}\\
y_{n-1}=-\frac{u_ny_n}{u_n+(2n+\alpha)y_n}.\label{KLym1}
\end{gather}
We note that~\eqref{KLauy} agrees with~\eqref{KLrea} and~\eqref{KLy} yields~\eqref{KLreb}, as stated in Theorem~\ref{theorem laguerrepainleve}.

Using~\eqref{KLA}, \eqref{KLH} and~\eqref{KLH0}, the Schlesinger equation~\eqref{OSL} amounts to
\begin{gather}
[H_\infty,A_\infty]=0,\label{SEL1}\\
\dt{A_1}=[R_0,A_\infty]+[H_\infty,A_1],\label{SEL2}\\
\dt{A_0}=-R_0+[R_0,A_1]+[H_\infty,A_0],\label{SEL3}\\
[R_0,A_0]=0.\label{SEL4}
\end{gather}
Equations~\eqref{SEL1} and~\eqref{SEL4} are automatically satisfied. Entries $(1,1)$ and $(2,2)$ of~\eqref{SEL2} are automatically satisfied. Entry $(2,1)$ of~\eqref{SEL2} gives
\begin{gather}
(\alpha+1)t\dt{{h}_{n-1}}=-h_{n-1}y_{n-1},\label{dlh}
\end{gather}
and, after substitution of~\eqref{dlh} into entry $(1,2)$ of~\eqref{SEL2}, we get
\begin{gather}
(\alpha+1)t\dt{{a}_n}=a_n(y_{n-1}-y_n).\label{dla}
\end{gather}
Entries $(1,1)$ and $(2,2)$ of~\eqref{SEL3} give the same equation
\begin{gather}
(\alpha+1)t\dt{{u}_n}=-u_n+a_n(y_n-y_{n-1}).\label{dlu1}
\end{gather}
After substitution of~\eqref{dlh} into entry $(2,1)$ of~\eqref{SEL3}, we get
\begin{gather*}
(\alpha+1)t\dt{{y}_{n-1}}=-y_{n-1}^2+(2n+\alpha-1)y_{n-1}+2u_n,
\end{gather*}
and, after substitution of~\eqref{dlh} and~\eqref{dla} into entry $(1,2)$ of~\eqref{SEL3}, we obtain
\begin{gather}
(\alpha+1)t\dt{{y}_n}=y_n^2-(2n+\alpha+1)y_n-2u_n.\label{dly2}
\end{gather}
\begin{Lemma} \label{lemma laguerresode}
The function $y_n(t)$ satisfies the second-order differential equation
\begin{gather}
\dt[2]{y_n}=\frac{1}{y_n}\Big(\dt{y_n}\Big)^2-\frac{1}{t}\dt{y_n}-\frac{(2n+\alpha+1)y_n^2}{t^2(\alpha+1)^2}+\frac{y_n^3}{t^2(\alpha+1)^2}. \label{KLPy}
\end{gather}
Furthermore, by setting $y_n(t)=tq_n(t)$, the function $q_n(t)$ solves an integrable case of the Painlev\'e equation $P_{\rm III}$ \eqref{P3}, with the parameters $a$, $b$, $c$, $d$ given in~\eqref{KLP3}.
\end{Lemma}
\begin{proof}
Substituting~\eqref{KLauy} and~\eqref{KLym1} into~\eqref{dlu1} yields
\begin{gather}
t(\alpha+1)\dt{u_n}=-u_n\left(2n+\alpha+1+\frac{2u_n}{y_n}\right),\label{dlu2}
\end{gather}
which, together with~\eqref{dly2}, forms a~system of two first-order differential equations for $y_n$ and~$u_n$. Solving~\eqref{dly2} for $u_n$ yields
\begin{gather}
2u_n=y_n ^2-(2n+1+\alpha)y_n-(\alpha+1)t\dt{y_n},\label{KLuye}
\end{gather}
which corresponds to~\eqref{KLuy}. Finally, substituting~\eqref{KLuye} into~\eqref{dlu2} leads to the second-order differential equation~\eqref{KLPy}. By setting $y_n(t)=tq_n(t)$, a~straightforward calculation verifies that~$q_n$ satisfies the $P_{\rm III}$ equation~\eqref{P3} with the parameters $a$, $b$, $c$, $d$ stated in~\eqref{KLP3}.
\end{proof}
\begin{Lemma} \label{lemma laguerretaylor}
The function $y_n(t)$ is a~rational function of $t$. Its Taylor expansion around $t=\infty$ on the Riemann sphere is given by
\begin{gather}
y_n(t)=\frac{(\alpha+1)(\alpha+1)_n}{n! t}+O\left(\frac{1}{t^2}\right),\qquad \mbox{as}\quad  t\to \infty.\label{aKLPy}
\end{gather}
\end{Lemma}
\begin{proof}
From~\eqref{dlh}, we have
\begin{gather}
y_n(t)=-\frac{(\alpha+1)t}{h_n(t)}\dt{h_n(t)}.\label{ydlh}
\end{gather}
The moments for the Krall--Laguerre weight~\eqref{wdKLa} are
\begin{gather}
\mu_0=\nu_0+\frac{1}{t}, \qquad\mu_k=\nu_k,\qquad  k\geq 1,\qquad  \mbox{with}\quad  \nu_k=(\alpha+1)_k. \label{mklt}
\end{gather}
Let $p_n(x,t)$ denote the monic Krall--Laguerre-type polynomials, where the dependence on $\alpha$ is omitted for brevity. From the explicit expression~\eqref{opm} and using~\eqref{mklt}, it follows that
\begin{gather*}
p_n(x,t)=p_n(x,\infty)+\frac{r_{n-1}(x)}{t}+O\left(\frac{1}{t^2}\right),
\end{gather*}
where $p_n(x,\infty)=\lim_{t\to\infty} p_n(x,t)$ represent the monic generalized Laguerre polynomials, and ${r_{n-1}(x)}$ is a~polynomial of degree $n-1$. The orthogonality relations imply that
\begin{gather}
\int_0^\infty p_n^2(x,t) x^\alpha {\rm e}^{-x}\dd x=\int_0^\infty p_n^2(x,\infty) x^\alpha {\rm e}^{-x}\dd x+O\left(\frac{1}{t^2}\right).\label{chklt}
\end{gather}
Since the coefficients of the polynomial $p_n(x,t)$ are rational functions of $t$, $h_n(t)$ is a~rational function of $t$. Thus, by~\eqref{ydlh}, $y_n(t)$ is a~rational function of $t$. Using~\eqref{chklt} and standard formulas for the generalized Laguerre polynomials (see~\cite[equation (10.12)]{EMOT}), we obtain
\begin{align*}
h_n(t)&=\frac{1}{\Gamma(\alpha+1)}\int_0^\infty p_n^2(x,t) x^\alpha {\rm e}^{-x}\dd x+\frac{p_n^2(0,t)}{t},\\
&=\frac{1}{\Gamma(\alpha+1)}\int_0^\infty p_n^2(x,\infty) x^\alpha {\rm e}^{-x}\dd x+\frac{p_n^2(0,\infty)}{t}+O\left(\frac{1}{t^2}\right),\\
&=(\alpha+1)_n n!\left(1+\frac{(\alpha+1)_n}{n!t}+O\left(\frac{1}{t^2}\right)\right),
\end{align*}
which, combined with~\eqref{ydlh}, establishes~\eqref{aKLPy}.
\end{proof}
\begin{Proposition} \label{proposition laguerrefirstintegral}
The first integral of~\eqref{KLPy} reads
\begin{gather}
\frac{t^2}{y_n^2}\left(\dt{y_n}\right)^2-\frac{y_n^2}{(\alpha+1)^2}+\frac{2(2n+1+\alpha)y_n}{(\alpha+1)^2}=1.\label{FIPy}
\end{gather}
As a~consequence, for Krall--Laguerre-type orthogonal polynomials, the differential equation~\eqref{lae} is given by~\eqref{KLde} as announced in Theorem~{\rm\ref{theorem laguerrepainleve}}, and reduces to the standard differential equation~\eqref{Lde} satisfied by the generalized Laguerre polynomials as $t\to\infty$.
\end{Proposition}
\begin{proof}
Substituting $q_n=y_n/t$ in the first integral~\eqref{FIP32} of~\eqref{P3} with parameters as in~\eqref{KLP3}, we obtain that the left-hand side of~\eqref{FIPy} is a~constant $C$. From~\eqref{aKLPy}, by taking the limit as $t\to\infty$, it follows that $C=1$, which establishes~\eqref{FIPy}. Substituting~\eqref{UVWklt}, \eqref{KLtheta}, \eqref{KLomega} into~\eqref{K}, using~\eqref{dxklt} and~\eqref{dx2klt}, we obtain
\begin{gather}
K_n=x\bigl(n(n+\alpha+y_n-x)-a_n-u_n\bigr)+y_n\left(a_n+u_n-n(n+\alpha+1)-\frac{u_n}{y_n}\right).\label{Kklt}
\end{gather}
Using~\eqref{KLauy}, \eqref{KLuye} and~\eqref{FIPy} gives
\begin{gather}
a_n+u_n-n(n+\alpha+1)-\frac{u_n}{y_n}\nonumber\\
\qquad=-\frac{(\alpha+1)^2}{4}\left\{\frac{t^2}{y_n^2}\left(\dt{y_n}\right)^2-\frac{y_n^2}{(\alpha+1)^2}+\frac{2(2n+1+\alpha)y_n}{(\alpha+1)^2}-1\right\}=0.\label{Kklt0}
\end{gather}
Hence, from this equation and~\eqref{KLuye}, we have
\begin{gather}
K_n=x\left(n(y_n-1-x)-\frac{u_n}{y_n}\right)=\frac{x}{2}\Big\{(\alpha+1)\left(\frac{t}{y_n}\dt{y_n}+1\right)+(2n-1)y_n-2nx\Big\}.\label{Kkltf}
\end{gather}
Remembering~\eqref{lae}, using~\eqref{UVWklt}, \eqref{KLtheta} and~\eqref{Kkltf}, \eqref{KLde} follows. As $t\to\infty$, using again~\eqref{aKLPy}, this equation reduces to the standard equation~\eqref{Lde} satisfied by the generalized Laguerre polynomials, which completes the proof.
\end{proof}
\begin{Remark} \label{remark laguerre}
For later use in Section \ref{sec5}, we observe from~\eqref{Kklt} and~\eqref{Kklt0} that the polynomial~$K_n$ is divisible by $x$ if and only if~\eqref{FIPy} is satisfied.
\end{Remark}

Combining Lemma~\ref{lemma laguerretaylor} with the integration of~\eqref{FIPy} yields an explicit formula for $y_n(t)$, as stated in the next proposition. In particular, it shows that the solution $y_n(t)$ of~\eqref{KLPy} is uniquely determined by the asymptotic behavior~\eqref{aKLPy}, as stated in~\eqref{aKLP3}. This completes the proof of Theorem~\ref{theorem laguerrepainleve}, establishing that Krall--Laguerre-type polynomials are fully characterized by the $P_{\rm III}$ family of solutions $q_{n}(t)=\frac{y_n(t)}{t}$, $n\geq 0$, $t\in\mathbb{R}$, $t>0$.
\begin{Proposition} \label{proposition laguerresolution}
Let
\begin{gather}
\tau_n(t)=n! t+(\alpha+2)_n,\qquad  n\geq 0.\label{tauklt}
\end{gather}
Then,
\begin{gather}
y_0(t)=\frac{\alpha+1}{t+1}, \qquad y_n(t)=\frac{(n-1)!(\alpha+1)^2(\alpha+2)_{n-1}t}{\tau_{n-1}(t)\tau_n(t)},\qquad  n\geq 1.\label{KLyn}
\end{gather}
\end{Proposition}
\begin{proof}
From~\eqref{FIPy}, an elementary quadrature yields
\begin{gather}
y_n(t)=\frac{k(\alpha+1)^2t}{(kt+1)(k(\alpha+n+1)t+n)},\label{SFIPy}
\end{gather}
where $k\neq 0$ is an arbitrary constant. The asymptotic behavior~\eqref{aKLPy} implies that
\begin{gather}
k=\frac{n!}{(\alpha+2)_n}.\label{kSFIPy}
\end{gather}
Substituting~\eqref{kSFIPy} into~\eqref{SFIPy} yields~\eqref{KLyn}, with $\tau_n(t)$ defined as in~\eqref{tauklt}.
\end{proof}

Finally, we obtain a~system of differential equations for the recurrence coefficients in terms of the Toda lattice hierarchy and its master symmetries.
\begin{Corollary}
The recurrence coefficients $a_n(t)$ and $b_n(t)$ for the Krall--Laguerre-type polynomials satisfy the following system of differential equations
\begin{align}
(\alpha+1)t\dt{a_n}&=(T_1-V_0)a_n,\label{KLTVa}\\
(\alpha+1)t\dt{b_n}&= (T_1-V_0)b_n,\label{KLTVb}
\end{align}
where $T_1$ and $V_0$ are the vector fields defined in~\eqref{T1} and~\eqref{V0}.
\end{Corollary}
\begin{proof}
It follows immediately from~\eqref{KLy} and~\eqref{dla} that
\begin{gather*}
(\alpha+1)t\dt{a_n}=a_n(b_n-b_{n-1})-2a_n,
\end{gather*}
which, recalling~\eqref{T1} and~\eqref{V0}, yields~\eqref{KLTVa}. From~\eqref{c} and~\eqref{KLu}, we have
\begin{gather}
c_n=-a_n-u_n,\label{kltcau}
\end{gather}
hence, from~\eqref{dla} and~\eqref{dlu1}, we obtain
$
(\alpha+1)t\dt{c_n}=u_n$.
From~\eqref{c} and~\eqref{kltcau}, we deduce that
\begin{gather*}
(\alpha+1)t\dt{b_n}=u_n-u_{n+1}=-a_n-c_n+a_{n+1}+c_{n+1}=a_{n+1}-a_n-b_n,
\end{gather*}
which, recalling~\eqref{T1} and~\eqref{V0}, establishes~\eqref{KLTVb} and completes the proof.
\end{proof}

\section[Krall-Jacobi-type polynomials and P\_V]{Krall--Jacobi-type polynomials and $\boldsymbol{P_{\rm V}}$}\label{sec4}
In this section, we prove case (b) of Theorem~\ref{theorem krallpainleve}. Throughout the section, we set
\begin{gather}
\gamma_n=\alpha+\beta+2n+1,\qquad \mbox{for} \quad n=0,1,2,\ldots,\label{KJga}
\end{gather}
where $\alpha, \beta \in \mathbb{R}$ and $\alpha, \beta>-1$. Furthermore, as in the previous section, the prime symbol $'$ denotes $\mathrm{d}/\mathrm{d}x$. More precisely, we shall establish the following theorem.
\begin{Theorem} \label{theorem jacobipainleve}
The Krall--Jacobi-type polynomials with the weight distribution defined as in equation~\eqref{wdKJ} are completely characterized by a~sequence of rational functions $y_n(t)=1-q_n(t)$, $n\geq 0$. Here, $q_n(t)$ is a~solution of the integrable case of the $P_{\rm V}$ equation~\eqref{P5}, with the parameters
\begin{gather}
a=\frac{\gamma_n^2}{2(\beta+1)^2},\qquad b=-\frac{\alpha^2}{2(\beta+1)^2},\qquad c=0,\qquad d=0,\label{KJP5}
\end{gather}
where $\gamma_n$ is given by~\eqref{KJga}. This solution is uniquely determined by the asymptotic behavior
\begin{gather}
q_n(t)=1-\frac{(\beta+1)(\beta+1)_n(\alpha+\beta+2)_{n-1}}{n!(\alpha+1)_n t}+O\left(\frac{1}{t^2}\right),\qquad \mbox{as} \quad t\to \infty.\label{aKJP5}
\end{gather}
The coefficients of the recurrence relation~\eqref{re} satisfied by these polynomials are given by
\begin{align}
a_n&=-\frac{(\gamma_n-1)^2u_n(2v_ny_n+u_n)}{4\gamma_{n-1}\gamma_n y_n^2},\label{KJrea}\\
b_n&=\frac{(1-\gamma_n)v_n+1-\gamma_ny_n}{\gamma_n+1},\label{KJreb}
\end{align}
where
\begin{align}
u_n={}&\frac{\alpha^2y_n}{2(\gamma_n-1)^2(y_n-1)}-\frac{1}{2(\gamma_n-1)^2y_n(y_n-1)}\left(t(\beta+1)\dt{y_n}+\gamma_ny_n(1-y_n)\right)^2,\label{KJuy}\\
v_n={}&-\frac{(\beta+1)t}{(\gamma_n-1)^2y_n}\dt{y_n}-\frac{(\beta+1)^2t^2}{2(\gamma_n-1)^2y_n^2(1-y_n)}\left(\dt{y_n}\right)^2+\frac{\alpha^2}{2(\gamma_n-1)^2(1-y_n)}\nonumber\\
&  -\frac{\gamma_n(\gamma_n-2)(y_n+1)+2}{2(\gamma_n-1)^2}\label{KJvy}.
\end{align}
The Laguerre differential equation~\eqref{lae} satisfied by these polynomials is given by
\begin{gather}
x(x-1)(x-y_n)\left\{g_n''+\left(\frac{\beta+2}{x}+\frac{\alpha+1}{x-1}-\frac{1}{x-y_n}\right)g'_n\right\}\nonumber\\
\qquad+\frac{1}{2}\biggl\{(\beta+1)\left(\frac{t}{y_n}\dt{y_n}+1\right)+\bigl((2n-1)(\alpha+\beta+1)+2n^2\bigr)y_n \nonumber\\
\qquad\phantom{+}{} -2n(\alpha+\beta+n+1)x\biggr\}g_n=0.\label{KJde}
\end{gather}
As $t\to\infty$, the coefficients~\eqref{KJrea} and~\eqref{KJreb} reduce to the recurrence coefficients of the Jacobi polynomials, and the differential equation~\eqref{KJde} reduces to the standard hypergeometric differential equation for these polynomials
\begin{gather}
x(x-1)g_n''+((\alpha+\beta+2)x-\beta-1)g_n'-n(\alpha+\beta+n+1)g_n=0.\label{Jde}
\end{gather}
\end{Theorem}
\begin{Remark}
The linear second-order differential equation~\eqref{KJde} features four regular singular points (including the point at infinity) and is thus of the Heun type. The case \(\beta = 0\) corresponds to the Jacobi-type polynomials in Krall’s classification~\cite{K2}, for which equation~\eqref{KJde} was first derived by Littlejohn and Shore~\cite{LLS} using a~fourth-order differential operator with these polynomials as eigenfunctions. As noted in the introduction, Magnus, Ndayiragije, and Ronveaux~\cite{MNR} also obtained equation~\eqref{KJde} via Laguerre’s theory, though without establishing a~connection to the Painlev\'e V equation. They established that Krall--Jacobi-type polynomials are the only families of orthogonal polynomials satisfying Heun's differential equation -- a~notable result based on a~proof by W.~Hahn.
\end{Remark}

Theorem~\ref{theorem jacobipainleve} will be proved following a~scheme similar to the one used in the previous section.
\begin{Lemma}\!
The Krall--Jacobi-type polynomials are semi-classical; precisely, the Stieltjes transform $f$ \eqref{Stm} satisfies the equation
\begin{gather}
Wf'=2Vf+U,\label{kjsc}
\end{gather}
with
\begin{gather}
W=x^2(1-x),\qquad  2V=x\bigl(\beta-(\alpha+\beta)x\bigr),\qquad
U=\mu_0(\alpha+\beta+1)x-\frac{\beta+1}{t}.\label{kjUVW}
\end{gather}
\end{Lemma}
\begin{proof}
By analytic continuation, it suffices to establish the result for the power series~\eqref{Stm}, which converges for $\vert x\vert >1$. From~\eqref{wdKJ}, we have
\begin{gather*}
\mu_0=c\int_0^1 x^\beta (1-x)^{\alpha}\dd x+\frac{1}{t},\qquad \mu_k=c\int_0^1 x^{\beta+k}(1-x)^\alpha\dd x, \qquad k\geq 1,
\end{gather*}
where $c$ is a~constant irrelevant to the argument. Integrating by parts yields
\begin{gather*}
\mu_{k+1}=\frac{c(\beta+k+1)}{(\alpha+1)} \int_0^1x^{\beta+k}(1-x)^{\alpha}(1-x)\dd x,\qquad  k\geq 0,
\end{gather*}
which implies
\smash{$
(\alpha+\beta+k+2)\mu_{k+1}=(\beta+k+1)\mu_k-\frac{\beta+1}{t}\delta_{k0}$}, $  k\geq 0$,
where $\delta_{k0}$ denotes the Kronecker delta. Rewriting this equation as
$
-(k+1)\mu_k=\beta\mu_k-(\alpha+\beta)\mu_{k+1}-(k+2)\mu_{k+1}-\smash{\frac{\beta+1}{t}\delta_{k0}}$,
we obtain
\begin{gather*}
f'=\frac{\beta}{x}f+(\alpha+\beta)\left(\frac{\mu_0}{x}-f\right)+x\left(f'+\frac{\mu_0}{x^2}\right)-\frac{\beta+1}{tx^2},
\end{gather*}
which leads to~\eqref{kjsc} with $U$, $V$ and $W$ defined as in~\eqref{kjUVW}.
\end{proof}

From~\eqref{ctheta}, \eqref{comega} and~\eqref{kjUVW}, using~\eqref{abch}, \eqref{c} and~\eqref{d}, we get
\begin{align}
\Theta_n&=-\gamma_n(x-y_n),\label{KJtheta}\\
\Omega_n&=\frac{(1-\gamma_n)\bigl(x^2+v_nx+u_n\bigr)}{2}\label{KJomega},
\end{align}
with
\begin{gather}
\gamma_ny_n=-(1+\gamma_n)b_n+\beta+2c_n+2n+1,\label{ybc}\\
\frac{(1-\gamma_n)v_n}{2}=\frac{\beta}{2}+c_n+n,\label{vc}\\
\frac{(1-\gamma_n)u_n}{2}=-\gamma_n a_n-c_n(c_n+1)+2d_n.\label{uacd}
\end{gather}

With the notations introduced in~\eqref{KJtheta} and~\eqref{KJomega}, using~\eqref{kjUVW}, the system~\eqref{lmde} reads
\begin{gather*}
Z'=\left(\frac{A_0}{x^2}+\frac{A_1}{x}+\frac{A_2}{x-1}\right)Z,
\end{gather*}
with
\begin{align}
&A_0=\begin{pmatrix} -\dfrac{(\gamma_n-1)u_n}{2}&-\gamma_na_ny_n\\ \gamma_{n-1}y_{n-1}&\dfrac{(\gamma_n-1)u_n}{2}\end{pmatrix},\nonumber\\
& A_1=\begin{pmatrix}-\dfrac{(\gamma_n-1)(u_n+v_n)+\beta}{2}&-\gamma_na_n(y_n-1)\\\gamma_{n-1}(y_{n-1}-1)& \dfrac{(\gamma_n-1)(u_n+v_n)-\beta}{2}\end{pmatrix},\nonumber\\
&A_2=-A_1+\begin{pmatrix} \dfrac{\gamma_n-1-\alpha-\beta}{2}&0\\0&-\dfrac{\gamma_n-1+\alpha+\beta}{2}\end{pmatrix}.\label{KJA2}
\end{align}
From~\eqref{H} with $a=0$, we obtain
$
H=\frac{R_0}{x}+H_\infty$,
and, using~\eqref{Ra} (with $a=0$ and $t_1=t$), \eqref{cRa}, \eqref{kjUVW}, \eqref{KJtheta} and~\eqref{KJomega}, we find
\begin{gather}
R_0=\frac{A_0}{(\beta+1)t},\qquad H_\infty=\frac{1}{h_{n-1}}\begin{pmatrix} 0&0\\0&\dfrac{{\rm d}h_{n-1}}{{\rm d}t}\end{pmatrix}.\label{KJH0I}
\end{gather}

The polynomial on the left-hand side of~\eqref{Kdiv} must be divisible by $W=x^2(1-x)$. Expressing it is divisible by $x^2$ leads to the two relations
\begin{gather}
y_{n-1}=-\frac{u_ny_n}{2v_ny_n+u_n},\label{jc1}\\
a_n=-\frac{(\gamma_n-1)^2u_n(2v_ny_n+u_n)}{4\gamma_{n-1}\gamma_n y_n^2},\label{jc2}
\end{gather}
and imposing it is divisible by $1-x$, using~\eqref{jc1} and~\eqref{jc2}, leads to a~third relation
\begin{gather}
(\gamma_n-1)^2(u_n+v_n y_n)^2=\bigl(\alpha^2-(\gamma_n-1)^2-2(\gamma_n-1)^2(u_n+v_n)\bigr)y_n^2.\label{jc3}
\end{gather}
Equation~\eqref{jc2} corresponds to~\eqref{KJrea}, while~\eqref{ybc} and~\eqref{vc} lead to~\eqref{KJreb}.

The Schlesinger equation~\eqref{OSL} amounts to
\begin{gather}
[R_0,A_0]=0,\label{SEJ0}\\
\dt{A_0}=-R_0+[R_0,A_1]+[H_\infty,A_0],\label{SEJ1}\\
\dt{A_1}=[H_\infty,A_1]+[A_2,R_0],\label{SEJ2}\\
\dt{A_2}=[H_\infty,A_2]+[R_0,A_2].\label{SEJ3}
\end{gather}
From~\eqref{KJH0I}, equation~\eqref{SEJ0} is automatically satisfied. Using~\eqref{KJA2}, equation~\eqref{SEJ3} is identical to equation~\eqref{SEJ2}. From entries $(1,1)$ (or $(2,2)$) of~\eqref{SEJ1} and~\eqref{SEJ2}
we obtain
\begin{gather}
(\beta+1)t\dt{u_n}=-u_n+\frac{2\gamma_{n-1}\gamma_na_n(y_{n-1}-y_n)}{\gamma_n-1},\qquad
(\beta+1)t\dt{v_n}=u_n.\label{djv}
\end{gather}
Solving entries $(2,1)$ of~\eqref{SEJ1} and~\eqref{SEJ2} for $\dt{h_{n-1}}$ and $\dt{y_{n-1}}$, we obtain
\begin{gather}
(\beta+1)t\dt{h_{n-1}}=(2-\gamma_n)h_{n-1}y_{n-1},\nonumber\\
(\beta+1)t\dt{y_{n-1}}=(1-\gamma_n)(u_n+v_ny_{n-1})+(2-\gamma_n)y_{n-1}^2-y_{n-1}.\label{djh}
\end{gather}
Substituting~\eqref{djh} into entries $(1,2)$ of~\eqref{SEJ1} and~\eqref{SEJ2} allows us to solve for $\dt{a_n}$ and $\dt{y_n}$,
\begin{gather}
(\beta+1)t\dt{a_n}=a_n\bigl(\gamma_n(y_{n-1}-y_n)-2y_{n-1}\bigr),\label{dja}\\
(\beta+1)t\dt{y_n}=(\gamma_n-1)(u_n+v_ny_n)+\gamma_ny_n^2-y_n.\label{djy}
\end{gather}
\begin{Lemma}
The function $y_n(t)$ solves the second-order differential equation
\begin{gather}
\dt[2]{y_n}=\left(\frac{1}{y_n}+\frac{1}{2(y_n-1)}\right)\Big(\dt{y_n}\Big)^2-\frac{1}{t}\dt{y_n}+
\frac{y_n^2}{2t^2(\beta+1)^2}\left(\gamma_n^2(y_n-1)-\frac{\alpha^2}{y_n-1}\right).\!\!\!\label{pj}
\end{gather}
Setting $y_n(t)=1-q_n(t)$, the function $q_n(t)$ solves an integrable case of the Painlev\'e equation~$P_{\rm V}$ \eqref{P5}, with the parameters $a$, $b$, $c$, $d$ given in~\eqref{KJP5}.
\end{Lemma}
\begin{proof}
By rewriting~\eqref{djy} in the form
\begin{gather}
(\beta+1)t\dt{y_n}-\gamma_ny_n^2+y_n=(\gamma_n-1)(u_n+v_ny_n), \label{uvy1}
\end{gather}
squaring both sides, and using~\eqref{jc3}, we obtain
\begin{gather}
\left((\beta+1)t\dt{y_n}-\gamma_ny_n^2+y_n\right)^2=\bigl(\alpha^2-(\gamma_n-1)^2-2(\gamma_n-1)^2(u_n+v_n)\bigr)y_n^2.\label{uvy2}
\end{gather}
Solving~\eqref{uvy1} and~\eqref{uvy2} for $u_n$ and $v_n$ yields~\eqref{KJuy} and~\eqref{KJvy}. Substituting these expressions into~\eqref{djv} leads to~\eqref{pj} after a~lengthy but straightforward calculation. Finally, setting $y_n(t)=1-q_n(t)$, it is readily verified that $q_n(t)$ satisfies~\eqref{P5} with the parameters $a$, $b$, $c$, $d$ specified in~\eqref{KJP5}.
\end{proof}
\begin{Lemma} \label{lemma jacobitaylor}
The function $y_n(t)$ is a~rational function of $t$, whose Taylor expansion around $t=\infty$ on the Riemann sphere is given by
\begin{gather}
y_n(t)=\frac{(\beta+1)(\beta+1)_n(\alpha+\beta+2)_{n-1}}{n!(\alpha+1)_n t}+O\left(\frac{1}{t^2}\right),\qquad \mbox{as} \quad t\to\infty.\label{aKJy}
\end{gather}
\end{Lemma}
\begin{proof}
Equation~\eqref{djh} yields
\begin{gather}
y_n(t)=-\frac{(\beta+1)t}{(\gamma_{n+1}-2)h_n(t)}\dt{h_n(t)}.\label{djhy}
\end{gather}
Let $p_n(x,t)$ denote the Krall--Jacobi-type polynomials (omitting their dependence on $\alpha$ and~$\beta$), and let $p_n(x,\infty)=\lim_{t\to\infty} p_n(x,t)$ be the standard Jacobi polynomials, with parameters $\alpha,\beta >-1$, normalized to be monic on the interval $[0, 1]$. Following the same argument as in the proof of Lemma~\ref{lemma laguerretaylor}, we deduce that $h_n(t)$ is a~rational function of $t$ satisfying
\begin{align}
h_n(t)&=\frac{\Gamma(\alpha+\beta+2)}{\Gamma(\alpha+1)\Gamma(\beta+1)}\int_0^1 p_n^2(x,\infty) x^\beta(1-x)^{\alpha}\dd x+\frac{p_n^2(0,\infty)}{t}+O\left(\frac{1}{t^2}\right).\label{ahy}
\end{align}
Setting $h_n(\infty)=\lim_{t\to\infty} h_n(t)$, standard formulas for the Jacobi polynomials (see~\cite[equation~(10.8)]{EMOT}) imply that
\begin{gather*}
h_n(\infty)=\frac{n!(\alpha+1)_n(\beta+1)_n}{(\alpha+\beta+2n+1)(\alpha+\beta+2)_{n-1}[(\alpha+\beta+n+1)_n]^2},
\end{gather*}
and
\begin{gather*}
p_n(0,\infty)=\frac{(-1)^n(\beta+1)_n}{(\alpha+\beta+n+1)_n}.
\end{gather*}
The result~\eqref{aKJy} then follows directly from~\eqref{djhy} using~\eqref{ahy}.
\end{proof}
\begin{Proposition}
The first integral of~\eqref{pj} reads
\begin{gather}
\frac{t^2}{(1-y_n)y_n^2}\left(\dt{y_n}\right)^2-\frac{\alpha^2}{(\beta+1)^2(1-y_n)}+\frac{\gamma_n^2y_n}{(\beta+1)^2}=1-\frac{\alpha^2}{(\beta+1)^2}.\label{FIpj}
\end{gather}
Consequently, for Krall--Jacobi-type orthogonal polynomials, the differential equation~\eqref{lae} is given by~\eqref{KJde} as announced in Theorem~{\rm\ref{theorem jacobipainleve}}, and reduces to the standard differential equation~\eqref{Jde} satisfied by the Jacobi polynomials as $t\to\infty$.
\end{Proposition}
\begin{proof}
Substituting $q_n=1-y_n$ into the first integral~\eqref{FIP5} of $P_{\rm V}$ with parameters $a$, $b$, $c$, $d$ given in~\eqref{KJP5}, shows that the left-hand side of~\eqref{FIpj} is a~constant $C$. Taking the limit as $t \to\infty$ and using~\eqref{aKJy} yields
\smash{$C=1-\frac{\alpha^2}{(\beta+1)^2}$}. Next, substituting~\eqref{kjUVW}, \eqref{KJtheta} and~\eqref{KJomega} into~\eqref{K}, while employing~\eqref{jc1}, \eqref{jc2} and~\eqref{jc3}, one obtains
\begin{align}
K_n={}&\frac{\gamma_n x}{2}\big\{2n(\alpha+\beta+n+2)y_n-(\gamma_n-1)^2v_n-2n(n+\alpha+\beta)-\beta(\alpha+\beta)\nonumber\\
&-2n(\alpha+\beta+n+1)x\big\}\nonumber\\
&+\frac{\gamma_ny_n}{2}\left\{(\alpha+\beta+2n)\left(\frac{u_n}{y_n}+\gamma_nv_n\right)+\beta(\alpha+\beta+1)+2n(\alpha+\beta+n)\right\}.\label{Kkjt}
\end{align}
It follows from~\eqref{KJuy} and~\eqref{KJvy} that
\begin{gather}
(\alpha+\beta+2n)\left(\frac{u_n}{y_n}+\gamma_nv_n\right)+\beta(\alpha+\beta+1)+2n(\alpha+\beta+n)\nonumber\\
\qquad=-\frac{(\beta+1)^2}{2}\left\{\frac{t^2}{(1-y_n)y_n^2}\left(\dt{y_n}\right)^2
-\frac{\alpha^2}{(\beta+1)^2(1-y_n)}+\frac{\gamma_n^2y_n}{(\beta+1)^2}+\frac{\alpha^2}{(\beta+1)^2}-1\right\}\nonumber\\
\qquad =0,\label{Kkjt0}
\end{gather}
where we have used~\eqref{FIpj}. Using~\eqref{FIpj} once more to eliminate \smash{$\bigl(\dt{y_n}\bigr)^2$} in~\eqref{KJvy}, we obtain
\begin{gather*}
-(\gamma_n-1)^2v_n=\frac{(\beta+1)t}{y_n}\dt{y_n}-\gamma_ny_n+\beta(\alpha+\beta+1)+2n(\alpha+\beta+n)+1,
\end{gather*}
which implies
\begin{align*}
K_n={}&\frac{\gamma_n x}{2}\biggl\{(\beta+1)\left(\frac{t}{y_n}\dt{y_n}+1\right)+\bigl((2n-1)(\alpha+\beta+1)+2n^2\bigr)y_n \nonumber\\
& {}-2n(\alpha+\beta+n+1)x\biggr\}.
\end{align*}
Consequently, \eqref{KJde} follows directly from~\eqref{lae}, \eqref{kjUVW}, \eqref{KJtheta} and~\eqref{KJomega}. Finally, taking the limit as $t\to\infty$ and invoking~\eqref{aKJy}, the equation~\eqref{KJde} reduces to~\eqref{Jde}.
\end{proof}
\begin{Remark} \label{remark jacobi}
As already observed in Remark~\ref{remark laguerre}, we point out that it follows from~\eqref{Kkjt} and~\eqref{Kkjt0} that the polynomial $K_n$ is divisible by $x$ if and only if~\eqref{FIpj} is satisfied.
\end{Remark}

Combining Lemma~\ref{lemma jacobitaylor} with the integration of~\eqref{FIpj} yields an explicit formula for $y_n(t)$, as stated in the next proposition. In particular, it shows that the solution $y_n(t)$ of~\eqref{pj} is uniquely determined by the asymptotic behavior~\eqref{aKJy}, as stated in~\eqref{aKJP5}. This completes the proof of Theorem~\ref{theorem jacobipainleve}, confirming that Krall--Jacobi-type polynomials are fully described by the family of $P_{\rm V}$ solutions $q_n(t)=1-y_n(t)$, $n\geq 0$, $t\in\mathbb{R}$, $t>0$.
\begin{Proposition}
Let
\begin{gather}
\tau_n(t)=n!(\alpha+1)_nt+(\alpha+\beta+2)_n(\beta+2)_n, \qquad  n\geq 0.\label{KJtau}
\end{gather}
Then,
\begin{align}
y_0(t)&=\frac{\beta+1}{(\alpha+\beta+1)(t+1)},\label{KJy0}\\
y_n(t)&=\frac{(n-1)!(\beta+1)(\alpha+1)_{n-1}(\beta+1)_n(\alpha+\beta+2)_{n-1}t}{\tau_{n-1}(t)\tau_{n}(t)},\qquad  n\geq 1.\label{KJyn}
\end{align}
\end{Proposition}
\begin{proof}
From~\eqref{FIpj}, where $\gamma_n$ is defined in~\eqref{KJga}, an elementary quadrature yields
\begin{gather*}
y_n(t)=\frac{k(\beta+1)^2 t}{[kt+1][k(\beta+n+1)(\alpha+\beta+n+1)t+n(\alpha+n)]},
\end{gather*}
where $k\neq 0$ is an arbitrary constant. The asymptotic behavior~\eqref{aKJy} yields
\begin{gather*}
k=\frac{n!(\alpha+1)_n}{(\beta+2)_n(\alpha+\beta+2)_n},
\end{gather*}
which leads to~\eqref{KJy0} and~\eqref{KJyn}, with $\tau_n(t)$ as in~\eqref{KJtau}.
\end{proof}
\begin{Corollary}
The recurrence coefficients $a_n(t)$ and $b_n(t)$ for the Krall--Jacobi-type polynomials satisfy the following system of differential equations
\begin{align}
(\beta+1)t\dt{a_n}&=((\alpha+\beta+2)T_1+V_1-V_0)a_n,\label{KJTVa}\\
(\beta+1)t\dt{b_n}&=((\alpha+\beta+2)T_1+V_1-V_0)b_n,\label{KJTVb}
\end{align}
where $T_1$, $V_0$ and $V_1$ are defined as in~\eqref{T1}, \eqref{V0} and~\eqref{V1}.
\end{Corollary}
\begin{proof}
Equation~\eqref{dja} can be rewritten as
\begin{gather*}
(\beta+1)t\dt{a_n}=a_n(\gamma_{n-1}y_{n-1}-\gamma_ny_n).
\end{gather*}
From~\eqref{c} and~\eqref{ybc}, we have
$
\gamma_{n-1}y_{n-1}-\gamma_ny_n=(\gamma_n+1)b_n-(\gamma_{n-1}-1)b_{n-1}-2$,
which, by virtue of~\eqref{T1}, \eqref{V0} and~\eqref{V1}, establishes~\eqref{KJTVa}. Next, from~\eqref{c} and~\eqref{vc}, it follows that
\begin{gather*}
b_n-1=\frac{(1-\gamma_n)v_n}{2}-\frac{(1-\gamma_{n+1})v_{n+1}}{2}.
\end{gather*}
Thus, from~\eqref{djv}, we deduce
\begin{gather*}
(\beta+1)t\dt{b_n}=\frac{(1-\gamma_n)u_n}{2}-\frac{(1-\gamma_{n+1})u_{n+1}}{2}.
\end{gather*}
Hence, using~\eqref{c} and~\eqref{uacd}, we obtain
\begin{align}
(\beta+1)t\dt{b_n}={}&(\alpha+\beta+2n+3)a_{n+1}
-(\alpha+\beta+2n+1)a_n\nonumber\\
&+2(d_n-d_{n+1})+c_{n+1}^2-c_n^2-b_n.\label{KJTVbacd}
\end{align}
From~\eqref{c} and~\eqref{d}, a~straightforward computation yields
\begin{gather*}
2(d_n-d_{n+1})+c_{n+1}^2-c_n^2-b_n=2a_n+b_n(b_n-1).
\end{gather*}
Therefore, taking~\eqref{T1}, \eqref{V0} and~\eqref{V1} into account, \eqref{KJTVbacd} matches~\eqref{KJTVb}.
\end{proof}

\section[Koornwinder's polynomials and a~new integrable Schlesinger system of partial differential equations]{Koornwinder's polynomials and a~new integrable Schlesinger\\ system of partial differential equations}\label{sec5}
In this section, we prove that Koornwinder's polynomials~\cite{Ko} are completely characterized by an \emph{integrable} Schlesinger system of partial differential equations. Throughout this section, we set
\begin{gather}
\gamma_n=\alpha+\beta+2n+1,\qquad \mbox{for} \quad n=0,1,2,\ldots,\label{Kga}
\end{gather}
where $\alpha, \beta \in \mathbb{R}$ with $\alpha, \beta>-1$. Some proofs involve technically complex computations. However, with the hints provided, the reader will have no difficulty verifying the results using any computer algebra system. Here, the prime symbol $'$ denotes differentiation with respect to $x$ for functions of $x$ that may also depend on the parameters $t_1$ and $t_2$.

\begin{Theorem} \label{theorem koornwinderschlesinger} %\label{theorem 5.1}
The Koornwinder polynomials, orthogonal with respect to the weight distribution~\eqref{wdK}, are completely characterized by a~sequence of functions $y_n(t_1,t_2,\alpha, \beta)$ and $z_n(t_1,t_2,\allowbreak \alpha,\beta)$, which are rational in $t_1$ and $t_2$ and satisfy the following system of partial differential equations
\begin{align}
&\frac{\partial^2 y_n}{\partial t_1^2}=\frac{1}{2y_n}\left(\frac{\partial y_n}{\partial t_1}\right)^2-\frac{1}{t_1}\frac{\partial y_n}{\partial t_1} +\frac{1}{z_n} \frac{\partial y_n}{\partial t_1} \frac{\partial z_n}{\partial t_1}
+ \frac{\gamma_n^2 y_nz_n^2}{32(\beta+1)^2t_1^2},\label{PKy1}\\
&\frac{\partial^2 z_n}{\partial t_1^2}=\frac{1}{z_n}\left(\frac{\partial z_n}{\partial t_1}\right)^2-\frac{1}{t_1}\frac{\partial z_n}{\partial t_1}+\frac{1}{y_n}\frac{\partial y_n}{\partial t_1} \frac{\partial z_n}{\partial t_1}
-\frac{z_n-4}{2y_n^2}\left(\frac{\partial y_n}{\partial t_1}\right)^2
+\frac{\gamma_n^2(z_n-4)z_n^2}{32(\beta+1)^2t_1^2},\label{PKz1}\\
&\frac{\partial^2 y_n}{\partial t_2^2}=\frac{1}{y_n}\left(\frac{\partial y_n}{\partial t_2}\right)^2-\frac{1}{t_2}\frac{\partial y_n}{\partial t_2}+\frac{1}{z_n}\frac{\partial y_n}{\partial t_2} \frac{\partial z_n}{\partial t_2}
-\frac{y_n-4}{2z_n^2}\left(\frac{\partial z_n}{\partial t_2}\right)^2
+\frac{\gamma_n^2(y_n-4)y_n^2}{32(\alpha+1)^2t_2^2},\label{PKy2}\\
&\frac{\partial^2 z_n}{\partial t_2^2}=\frac{1}{2z_n}\left(\frac{\partial z_n}{\partial t_2}\right)^2-\frac{1}{t_2}\frac{\partial z_n}{\partial t_2} +\frac{1}{y_n} \frac{\partial y_n}{\partial t_2} \frac{\partial z_n}{\partial t_2}
+ \frac{\gamma_n^2 y_n^2z_n}{32(\alpha+1)^2t_2^2},\label{PKz2}
\end{align}
together with
\begin{gather}
(\beta+1)t_1\frac{\partial y_n}{\partial t_1}=(\alpha+1)t_2\frac{\partial z_n}{\partial t_2},\label{PK1}\\
(\alpha+1)z_nt_2\frac{\partial y_n}{\partial t_2}+(\beta+1)y_nt_1\frac{\partial z_n}{\partial t_1}=(\beta+1)(y_n+z_n-4)t_1\frac{\partial y_n}{\partial t_1}-y_nz_n.\label{PK2}
\end{gather}
The solution is uniquely determined by the following asymptotic behavior as $t_1, t_2 \to \infty$
\begin{align}
&y_n(t_1,t_2,\alpha,\beta)= \frac{a}{t_2}+\frac{a^2\bigl(\beta^2-(\alpha+1)^2-\gamma_n^2\bigr)}{8(\alpha+1)^2t_2^2}-\frac{ab(\alpha+\beta+1)}{4 (\beta+1)t_1t_2}
 +O\left(\left(\frac{1}{t_1^2}+\frac{1}{t_2^2}\right)^\frac{3}{2}\right),\!\!\!\label{aKy2}\\
&z_n(t_1,t_2,\alpha,\beta)= \frac{b}{t_1}+\frac{b^2\bigl(\alpha^2-(\beta+1)^2-\gamma_n^2\bigr)}{8(\beta+1)^2t_1^2} -\frac{ab(\alpha+\beta+1)}{4(\alpha+1)t_1t_2}
 +O\left(\left(\frac{1}{t_1^2} +\frac{1}{t_2^2}\right)^\frac{3}{2}\right),\!\!\!\label{aKz2}
\end{align}
with $\gamma_n$ as in~\eqref{Kga}, and $a$, $b$ given by
\begin{gather}
a=\frac{4(\alpha+1)(\alpha+1)_n(\alpha+\beta+2)_{n-1}}{n!(\beta+1)_n}, \qquad b=\frac{4(\beta+1)(\beta+1)_n(\alpha+\beta+2)_{n-1}}{n!(\alpha+1)_n}.\label{ab}
\end{gather}
In fact, only four equations -- either~\eqref{PKy1} and~\eqref{PKz1}, or~\eqref{PKy2} and~\eqref{PKz2} -- together with~\eqref{PK1} and~\eqref{PK2} are required, since the remaining equations follow as a~consequence.
\end{Theorem}

We begin by establishing the system of partial differential equations~\eqref{PKy1}, \eqref{PKz1}, \eqref{PKy2}, \eqref{PKz2}, \eqref{PK1}, \eqref{PK2}, introducing in particular the definitions of $y_n$ and~$z_n$. Then, starting with Lemma~\ref{lemma koornwindertaylor}, we develop the necessary tools to prove the integrability of this system and to solve it via separation of variables in Theorem~\ref{theorem koornwindersolution}. Along the way, explicit expressions are obtained for the recurrence coefficients~\eqref{re}, as well as for the Laguerre equation~\eqref{lae}, in terms of $y_n$ and~$z_n$. As a~limiting case, Theorem~\ref{theorem gegenbauerpainleve} shows that the Krall--Gegenbauer-type polynomials with equal jumps are completely described by an integrable case of the $P_{\rm V}$ equation. Throughout this section, the dependence of $y_n$, $z_n$, and related quantities on $(t_1,t_2)$ or $(t_1, t_2, \alpha, \beta)$ will only be noted when required by the context.
\begin{Lemma}
The Koornwinder polynomials are semi-classical; precisely, the Stieltjes transform~$f$ \eqref{Stm} satisfies the equation
\begin{gather}
Wf'=2Vf+U,\label{scKK1}
\end{gather}
with
\begin{align}
W&=\bigl(x^2-1\bigr)^2,\qquad  2V=\bigl(x^2-1\bigr)((\alpha+\beta)x+\alpha-\beta),\nonumber\\
U&=-\frac{2(\beta+1)(1-x)}{t_1}-\frac{2(\alpha+1)(1+x)}{t_2}-\mu_0(\alpha+\beta+1)\bigl(x^2-1\bigr).\label{scKK2}
\end{align}
\end{Lemma}
\begin{proof}
By analytic continuation, it suffices to establish the result on the power series~\eqref{Stm} which converges for $\vert x\vert >1$. From~\eqref{wdK},
the moments defining the Koornwinder polynomials are given by
\begin{gather}
\mu_k=\nu_{k}+\frac{(-1)^k}{t_1}+\frac{1}{t_2},\qquad  k\geq 0,\label{mKK}
\end{gather}
with
\[
\nu_k=c\int_{-1}^1 x^k(1-x)^{\alpha}(1+x)^{\beta}\dd x,
\]
$c$ a~constant irrelevant to the argument. Since
$
(\alpha+\beta+k+2)\nu_{k+1}=(\beta-\alpha)\nu_k+k\nu_{k-1}$,
we have
\begin{gather}
(\alpha+\beta+k+2)\mu_{k+1}=(\beta-\alpha)\mu_k+k\mu_{k-1}+(-1)^{k+1}\frac{2(\beta+1)}{t_1}+\frac{2(\alpha+1)}{t_2},\nonumber\\
 k\geq 0.\label{rmKK}
\end{gather}
Rewriting~\eqref{rmKK} as
\begin{gather*}
-(k+2)\mu_{k+1}=(\alpha+\beta)\mu_{k+1}+(\alpha-\beta)\mu_k-k\mu_{k-1}
+(-1)^k\frac{2(\beta+1)}{t_1}-\frac{2(\alpha+1)}{t_2},
\end{gather*}
since, from~\eqref{Stm}, we have
\[
f'+\frac{\mu_0}{x^2}=-\sum_{k=0}^{\infty}\frac{(k+2)\mu_{k+1}}{x^{k+3}},
\]
we obtain
\begin{gather*}
f'+\frac{\mu_0}{x^2}=\frac{\alpha+\beta}{x}\left(f-\frac{\mu_0}{x}\right)+\frac{\alpha-\beta}{x^2}f+\frac{f'}{x^2}
+\frac{2(\beta+1)}{t_1 x^2(x+1)}-\frac{2(\alpha+1)}{t_2 x^2(x-1)},
\end{gather*}
which establishes~\eqref{scKK1} with $U$, $V$ and $W$ as in~\eqref{scKK2}.
\end{proof}

From~\eqref{mKK}, we observe that
$
\mu_k(t_2,t_1,\beta,\alpha)=(-1)^k\mu_k(t_1,t_2,\alpha,\beta)$.
Hence, using~\eqref{opm}, we deduce that
$
p_n(x,t_2,t_1,\beta,\alpha)=(-1)^np_n(-x,t_1,t_2,\alpha,\beta)$,
which, in light of~\eqref{re}, implies
\begin{gather}
a_n(t_2,t_1,\beta,\alpha)=a_n(t_1,t_2,\alpha,\beta),\qquad  b_n(t_2,t_1,\beta,\alpha)=-b_n(t_1,t_2,\alpha,\beta).\label{ptalbe}
\end{gather}

From~\eqref{ctheta} and~\eqref{scKK2}, using~\eqref{c} and~\eqref{d}, one computes that
\begin{gather}
\Theta_n=\gamma_n\left(x^2-1+y_n(t_1,t_2,\alpha,\beta)\frac{1+x}{2}+z_n(t_1,t_2,\alpha,\beta)\frac{1-x}{2}\right),\label{Ktheta}
\end{gather}
where
\begin{align}
\gamma_ny_n(t_1,t_2,\alpha,\beta)={}&(\alpha+\beta+2n+3)\bigl(a_n+a_{n+1}+b_n^2\bigr)-2b_nc_n+4\sum_{i=1}^{n-1}a_i+2\sum_{i=0}^{n-1}b_i^2\nonumber\\
&+2(\alpha+n)b_n-2c_{n+1}+\alpha-\beta-2n-1,\label{Kyab}
\end{align}
and, applying~\eqref{ptalbe}, one obtains
\begin{gather}
z_n(t_1,t_2,\alpha,\beta)=y_n(t_2,t_1,\beta,\alpha).\label{zy}
\end{gather}
This implies that
\begin{gather}
\gamma_n(z_n-y_n)=2\bigl(\beta-\alpha-(\alpha+\beta+2n)b_n+2c_{n+1}\bigr).\label{Kzmy}
\end{gather}

From~\eqref{comega} and~\eqref{scKK2}, using~\eqref{abch}, \eqref{c}, \eqref{d} and~\eqref{e}, one computes that
\begin{align}
\Omega_n={}&\frac{\gamma_n-1}{2}\left(\bigl(x^2-1\bigr)(x+u_n(t_1,t_2,\alpha,\beta))+v_n(t_1,t_2,\alpha,\beta)\frac{1+x}{2}\right.\nonumber\\
&\left.{}+w_n(t_1,t_2,\alpha,\beta)\frac{1-x}{2}\right),\label{Komega}
\end{align}
where
\begin{gather}
(\gamma_n-1)u_n(t_1,t_2,\alpha,\beta)=\alpha-\beta-2c_n,\label{uc}\\
\frac{(\gamma_{n}-1)v_n(t_1,t_2,\alpha,\beta)}{2}=(\alpha+\beta)a_n(b_{n-1}+b_n)+2\alpha a_n+2\sum_{i=1}^{n-1}a_i+3\sum_{i=1}^{n-1}a_i(b_{i-1}+b_i)\nonumber\\
\phantom{\frac{(\gamma_{n}-1)v_n(t_1,t_2,\alpha,\beta)}{2}=}{}
+2a_n\sum_{i=0}^{n-2}b_i
+\sum_{i=0}^{n-1}\bigl(b_i^3+b_i^2-b_i\bigr)\nonumber\\
\phantom{\frac{(\gamma_{n}-1)v_n(t_1,t_2,\alpha,\beta)}{2}=}{}+a_n\bigl((2n+3)b_{n-1}+(2n+2)b_n+2n+1\bigr)-n,\label{Kvab}
\end{gather}
and, applying~\eqref{ptalbe}, one obtains
\begin{gather}
w_n(t_1,t_2,\alpha,\beta)=-v_n(t_2,t_1, \beta,\alpha).\label{uva}
\end{gather}
It follows that
\begin{gather}
(\gamma_n-1)(w_n-v_n)=4\left(n-\gamma_na_n-2\sum_{i=1}^{n-1}a_i-\sum_{i=0}^{n-1}b_i^2\right).\label{Kwmv}
\end{gather}
Note already that~\eqref{uc} yields
\begin{gather}
2b_n=(\gamma_{n+1}-1)u_{n+1}-(\gamma_n-1)u_n. \label{Krb}
\end{gather}
Since the polynomial on the left-hand side of~\eqref{Kdiv} is divisible by $W$, imposing this condition with \smash{$W=\bigl(x^2-1\bigr)^2$} as in~\eqref{scKK2}, leads to the four relations
\begin{gather}
y_{n-1}=\frac{4v_n^2z_n(w_ny_n-v_nz_n)}{(w_n-v_n)(v_nz_n-w_ny_n)^2+4v_nw_n(w_ny_n-v_nz_n+4y_nz_n)},\label{Krym1}\\
z_{n-1}=\frac{4w_n^2y_n(w_ny_n-v_nz_n)}{(w_n-v_n)(v_nz_n-w_ny_n)^2+4v_nw_n(w_ny_n-v_nz_n+4y_nz_n)},\label{Krzm1}\\
a_n=-\frac{(\gamma_n-1)^2}{\gamma_{n-1}\gamma_n}\Bigg\{\frac{(w_n-v_n)(v_nz_n-w_ny_n)}{16y_nz_n}
+v_nw_n\left(\frac{1}{v_nz_n-w_ny_n}-\frac{1}{4y_nz_n}\right)\Bigg\},\label{Kra}\\
u_n=\frac{(y_n-z_n)(z_nv_n-y_nw_n)}{8y_nz_n}
+(z_nv_n+y_nw_n)\left(\frac{1}{2y_nz_n}-\frac{1}{z_nv_n-y_nw_n}\right)\label{Kru}.
\end{gather}

From~\eqref{lm}, using~\eqref{scKK2}, \eqref{Ktheta} and~\eqref{Komega}, it follows that~\eqref{lmde} reads
$
Z'=AZ$,
with
\begin{gather}
A=\frac{A_1}{(x+1)^2}+\frac{A_2}{(x-1)^2}+\frac{A_3}{x+1}+\frac{A_4}{x-1},\label{KA}
\end{gather}
and
\begin{gather*}
A_1=\frac{1}{4}\begin{pmatrix} \dfrac{(\gamma_n-1)w_n}{2}&-\gamma_na_nz_n\\
\gamma_{n-1}z_{n-1}&-\dfrac{(\gamma_n-1)w_n}{2}\end{pmatrix},\qquad
A_2=\frac{1}{4}\begin{pmatrix} \dfrac{(\gamma_n-1)v_n}{2}&-\gamma_na_ny_n\\
\gamma_{n-1}y_{n-1}&-\dfrac{(\gamma_n-1)v_n}{2}\end{pmatrix},\\
A_3=\frac{1}{8}\begin{pmatrix} -4\beta+\dfrac{(\gamma_n-1)(v_n+w_n-4u_n+4)}{2}&-\gamma_na_n(y_n+z_n-4)\\
\gamma_{n-1}(y_{n-1}+z_{n-1}-4)&-4\beta-\dfrac{(\gamma_n-1)(v_n+w_n-4u_n+4)}{2}\end{pmatrix},\\
A_4=\frac{1}{8}\begin{pmatrix} -4\alpha-\dfrac{(\gamma_n-1)(v_n+w_n-4u_n-4)}{2}&\gamma_na_n(y_n+z_n-4)\\
-\gamma_{n-1}(y_{n-1}+z_{n-1}-4)&-4\alpha+\dfrac{(\gamma_n-1)(v_n+w_n-4u_n-4)}{2}\end{pmatrix}.
\end{gather*}

The Schlesinger equations~\eqref{PSL} are given by
\begin{align}
\frac{\partial A}{\partial t_1}&=\frac{\partial H_1}{\partial x}+[H_1,A],\label{SLEKK1}\\
\frac{\partial A}{\partial t_2}&=\frac{\partial H_2}{\partial x}+[H_2,A],\label{SLEKK2}
\end{align}
where $A$ is defined as in~\eqref{KA}, and
\begin{align*}
H_1&=\frac{A_1}{t_1(\beta+1)(x+1)}+\frac{1}{h_{n-1}}\begin{pmatrix} 0&0\\ 0&\dfrac{\partial h_{n-1}}{\partial t_1}\end{pmatrix},\\
H_2&=\frac{A_2}{t_2(\alpha+1)(x-1)}+\frac{1}{h_{n-1}}\begin{pmatrix} 0&0\\ 0&\dfrac{\partial h_{n-1}}{\partial t_2}\end{pmatrix}.
\end{align*}
These matrices are computed according to~\eqref{Hinf}, \eqref{cH1}, \eqref{Ra} and~\eqref{cRa} for $a=-1$, and analogously for $b=1$ (with $t_2$ replaced by $t_1$), while recalling from~\eqref{abch} that $h_n=a_nh_{n-1}$.

The Schlesinger equation~\eqref{SLEKK1} gives at the $(1,1)$ entry
\begin{gather}
(\beta+1)t_1\frac{\partial u_n}{\partial t_1}=-\frac{w_n}{4},\label{SKu1}\\
(\beta+1)t_1\frac{\partial v_n}{\partial t_1}=-\frac{\gamma_{n-1}\gamma_na_n(y_{n-1}z_n-z_{n-1}y_n)}{4(\gamma_n-1)},\label{SKv1}\\
(\beta+1)t_1\frac{\partial w_n}{\partial t_1}=-w_n+ \frac{\gamma_{n-1}\gamma_n a_n(z_n-z_{n-1})}{\gamma_n-1}
-\frac{\gamma_{n-1}\gamma_na_n(y_{n-1}z_n-z_{n-1}y_n)}{4(\gamma_n-1)},\label{SKw1}
\end{gather}
then entry $(2,2)$ is automatically satisfied. At the $(2,1)$ entry, we get
\begin{gather}
(\beta+1)t_1\frac{\partial h_{n-1}}{\partial t_1}=-\frac{(\gamma_n-2)h_{n-1}z_{n-1}}{4},\nonumber\\
(\beta+1)t_1\frac{\partial y_{n-1}}{\partial t_1}=-\frac{(\gamma_n-2)y_{n-1}z_{n-1}}{4}+\frac{(\gamma_n-1)(z_{n-1}v_n-y_{n-1}w_n)}{8},\nonumber\\
(\beta+1)t_1\frac{\partial z_{n-1}}{\partial t_1}=-\frac{z_{n-1}\bigl((\gamma_n-2)z_{n-1}-2\gamma_n+6\bigr)}{4}-\frac{(\gamma_n-1)z_{n-1}u_n}{2}\nonumber\\
\phantom{(\beta+1)t_1\frac{\partial z_{n-1}}{\partial t_1}=}{}+\frac{(\gamma_n-1)w_n}{2}+\frac{(\gamma_n-1)(z_{n-1}v_n-y_{n-1}w_n)}{8},\label{SKh1}
\end{gather}
and at the $(1,2)$ entry, using~\eqref{SKh1}, we get
\begin{align}
(\beta+1)t_1\frac{\partial y_n}{\partial t_1}={}&\frac{\gamma_ny_nz_n}{4}-\frac{(\gamma_n-1)(z_nv_n-y_nw_n)}{8},\label{SKy1}\\
(\beta+1)t_1\frac{\partial z_n}{\partial t_1}={}&\frac{z_n(\gamma_nz_n-2\gamma_n-2)}{4}+\frac{(\gamma_n-1)z_nu_n}{2}\nonumber\\
&-\frac{(\gamma_n-1)w_n}{2}-\frac{(\gamma_n-1)(z_nv_n-y_nw_n)}{8},\label{SKz1}\\
(\beta+1)t_1\frac{\partial a_n}{\partial t_1}={}&-\frac{a_n\bigl(\gamma_n(z_n-z_{n-1})+2z_{n-1}\bigr)}{4}.\label{SKa1}
\end{align}

The Schlesinger equation~\eqref{SLEKK2} gives at the $(1,1)$ entry
\begin{gather}
(\alpha+1)t_2\frac{\partial u_n}{\partial t_2}=-\frac{v_n}{4},\nonumber\\
(\alpha+1)t_2\frac{\partial v_n}{\partial t_2}=-v_n+\frac{\gamma_{n-1}\gamma_na_n(y_{n-1}-y_n)}{\gamma_n-1}
-\frac{\gamma_{n-1}\gamma_na_n(y_{n-1}z_n-z_{n-1}y_n)}{4(\gamma_n-1)},\nonumber\\
(\alpha+1)t_2\frac{\partial w_n}{\partial t_2}=-\frac{\gamma_{n-1}\gamma_na_n(y_{n-1}z_n-z_{n-1}y_n)}{4(\gamma_n-1)},\label{SKu2}
\end{gather}
then entry $(2,2)$ is automatically satisfied. At the $(2,1)$ entry, we get
\begin{gather}
(\alpha+1)t_2\frac{\partial h_{n-1}}{\partial t_2}=-\frac{(\gamma_n-2)h_{n-1}y_{n-1}}{4},\nonumber\\
(\alpha+1)t_2\frac{\partial y_{n-1}}{\partial t_2}=-\frac{y_{n-1}\bigl((\gamma_n-2)y_{n-1}-2\gamma_n+6\bigr)}{4}+\frac{(\gamma_n-1)y_{n-1}u_n}{2}\nonumber\\
\phantom{(\alpha+1)t_2\frac{\partial y_{n-1}}{\partial t_2}=}{} -\frac{(\gamma_n-1)v_n}{2}+\frac{(\gamma_n-1)(z_{n-1}v_n-y_{n-1}w_n)}{8},\nonumber\\
(\alpha+1)t_2\frac{\partial z_{n-1}}{\partial t_2}=-\frac{(\gamma_n-2)y_{n-1}z_{n-1}}{4}+\frac{(\gamma_n-1)(z_{n-1}v_n-y_{n-1}w_n)}{8},\label{SKh2}
\end{gather}
and at the $(1,2)$ entry, after substituting~\eqref{SKh2}, we get
\begin{align}
(\alpha+1)t_2\frac{\partial y_n}{\partial t_2}={}&\frac{y_n\bigl(\gamma_ny_n-2\gamma_n-2)}{4}-\frac{(\gamma_n-1)y_nu_n}{2}\nonumber\\
&+\frac{(\gamma_n-1)v_n}{2}-\frac{(\gamma_n-1)(z_nv_n-y_nw_n)}{8},\label{SKy2}\\
(\alpha+1)t_2\frac{\partial z_n}{\partial t_2}={}&\frac{\gamma_ny_nz_n}{4}-\frac{(\gamma_n-1)(z_nv_n-y_nw_n)}{8},\label{SKz2}\\
(\alpha+1)t_2\frac{\partial a_n}{\partial t_2}={}&-\frac{a_n\bigl(\gamma_n(y_n-y_{n-1})+2y_{n-1}\bigr)}{4}.\label{SKa2}
\end{align}

One can solve~\eqref{Kru}, \eqref{SKy1} and~\eqref{SKz1} for $u_n$, $v_n$, $w_n$ in terms of $y_n$, $z_n$ and $\frac{\partial y_n}{\partial t_1}$, $\frac{\partial z_n}{\partial t_1}$, which gives
\begin{gather}
u_n = \frac{\gamma_n(\gamma_n-2)(y_n-z_n)}{4(\gamma_n-1)^2}-\frac{1}{(\gamma_n-1)^2}+\frac{2t_1(\beta+1)}{(\gamma_n-1)^2z_n}\left(\frac{\partial y_n}{\partial t_1}-\frac{\partial z_n}{\partial t_1}\right) \nonumber\\
\phantom{u_n = }{}+\frac{8t_1^2(\beta+1)^2}{(\gamma_n-1)^2y_nz_n^2}\frac{\partial y_n}{\partial t_1}\frac{\partial z_n}{\partial t_1} -\frac{4t_1^2(\beta+1)^2(y_n+z_n-4)}{(\gamma_n-1)^2y_n^2z_n^2}\left(\frac{\partial y_n}{\partial t_1}\right)^2,\label{uyz}\\
v_n=\frac{4t_1(\beta+1)\frac{\partial y_n}{\partial{t_1}}-\gamma_ny_nz_n}{4(\gamma_n-1)^2y_nz_n^2}\biggl\{-4t_1(\beta+1)(y_n+z_n-4)\frac{\partial y_n}{\partial t_1}+8t_1(\beta+1)y_n\frac{\partial z_n}{\partial t_1} \nonumber\\
\phantom{v_n=}{}
 +y_nz_n\bigl(\gamma_n(y_n-z_n)-4(\gamma_n-2)\bigr)\biggr\},\label{vyz}\\
w_n = \frac{4t_1(\beta+1)\frac{\partial y_n}{\partial{t_1}}-\gamma_ny_nz_n}{4(\gamma_n-1)^2y_n^2z_n}\biggl\{-4t_1(\beta+1)(y_n+z_n-4)\frac{\partial y_n}{\partial t_1}+8t_1(\beta+1)y_n\frac{\partial z_n}{\partial t_1} \nonumber\\
\phantom{w_n =}{} +\gamma_ny_nz_n(y_n-z_n+4)\biggr\}.\label{wyz}
\end{gather}

Equation~\eqref{PK1} follows directly from~\eqref{SKy1} and~\eqref{SKz2}, whereas equation~\eqref{PK2} follows from equation~\eqref{SKy1}, \eqref{SKz1} and~\eqref{SKy2}. Substituting~\eqref{Krym1}, \eqref{Krzm1} and~\eqref{Kra} into~\eqref{SKv1} yields
\begin{gather}
(\beta+1)t_1\frac{\partial v_n}{\partial t_1}=\frac{(\gamma_n-1)\bigl(y_n^2w_n^2-z_n^2v_n^2\bigr)}{16 y_nz_n}.\label{SKv1s}
\end{gather}
Then, substituting~\eqref{uyz}, \eqref{vyz}, \eqref{wyz} into~\eqref{SKu1} and~\eqref{SKv1s} results in two equations involving only $y_n$, $z_n$ and their partial derivatives with respect to $t_1$ up to second order. Solving this system for \smash{$\frac{\partial^2 y_n}{\partial t_1^2}$} and \smash{$\frac{\partial^2 z_n}{\partial t_1^2}$} leads to~\eqref{PKy1} and~\eqref{PKz1}. One can verify that~\eqref{SKw1} is then automatically satisfied. Equations~\eqref{PKy2} and~\eqref{PKz2} can be obtained in a~similar manner, or alternatively, derived from~\eqref{PKy1}, \eqref{PKz1}, \eqref{PK1} and~\eqref{PK2}.

Note that equations~\eqref{Krb} and~\eqref{Kra}, together with~\eqref{uyz}, \eqref{vyz} and~\eqref{wyz}, already imply that the recurrence coefficients $a_n$ and $b_n$ are determined by the functions $y_n$ and $z_n$. Consequently, it remains to find these functions explicitly.
\begin{Lemma} \label{lemma koornwindertaylor}
The functions $y_n(t_1,t_2)$ and $z_n(t_1,t_2)$ are rational functions of $(t_1, t_2)$, thus extending to meromorphic functions on the product of two Riemann spheres $\mathbb{P}^1(\mathbb{C})\times \mathbb{P}^1(\mathbb{C})$. They are holomorphic in a~neighborhood of $(\infty,\infty)\in \mathbb{P}^1(\mathbb{C})\times \mathbb{P}^1(\mathbb{C})$ and vanish at that point, with the following Taylor expansions as $(t_1,t_2)\to (\infty, \infty)$
\begin{align}
y_n(t_1,t_2)&=\frac{4(\alpha+1)(\alpha+1)_n(\alpha+\beta+2)_{n-1}}{n!(\beta+1)_n t_2}+O\left(\frac{1}{t_1^2}+\frac{1}{t_2^2}\right),\label{aKy}\\
z_n(t_1,t_2)&=\frac{4(\beta+1)(\beta+1)_n(\alpha+\beta+2)_{n-1}}{n!(\alpha+1)_n t_1}+O\left(\frac{1}{t_1^2}+\frac{1}{t_2^2}\right).\label{aKz}
\end{align}
\end{Lemma}
\begin{proof}
Throughout the proof, we denote by $p_n(x,t_1,t_2)$ the Koornwinder polynomials normalized to be monic (omitting the dependence on $\alpha$, $\beta$), and by $p_n(x, \infty,\infty)$ their limit as~${t_1,t_2\to\infty}$, and similarly for related quantities. From~\eqref{SKh2}, we have
\begin{gather}
y_n(t_1,t_2)=-\frac{4t_2(\alpha+1)}{\gamma_nh_n}\frac{\partial h_n}{\partial t_2}(t_1,t_2).\label{dKhy}
\end{gather}
Recalling the definition of $h_n$ \eqref{h} and the weight distribution defining the Koornwinder polynomials~\eqref{wdK}, we obtain
\begin{gather}
h_n(t_1,t_2)=c\int_{-1}^1 p_n^2(x,t_1,t_2)(1-x)^\alpha(1+x)^\beta\dd x+ \frac{p_n^2(-1,t_1,t_2)}{t_1}+\frac{p_n^2(1,t_1,t_2)}{t_2},\label{KKh}
\end{gather}
with
\begin{gather*}
c=\frac{\Gamma(\alpha+\beta+2)}{2^{\alpha+\beta+1}\Gamma(\alpha+1)\Gamma(\beta+1)}.
\end{gather*}
By expressing $p_n(x,t_1,t_2)$ in terms of the moments~\eqref{mKK} via the classical formula~\eqref{opm}, it is clear that $h_n(t_1,t_2)$ is a~rational function of $(t_1,t_2)$. Moreover, since $p_n(x,t_1,t_2)$ is monic, from~\eqref{opm} we also deduce that
\begin{gather*}
p_n(x,t_1,t_2)=p_n(x,\infty,\infty)+\frac{r_{n-1}(x)}{t_1}+\frac{s_{n-1}(x)}{t_2}+O\left(\frac{1}{t_1^2}+\frac{1}{t_2^2}\right),
\end{gather*}
with $p_n(x,\infty,\infty)$ the Jacobi polynomials, normalized to be monic, and $r_{n-1}(x)$, $s_{n-1}(x)$ certain polynomials in $x$ of degree $n-1$. Consequently, by invoking the orthogonality of the Jacobi polynomials, it follows from~\eqref{KKh} that
\begin{gather}
h_n(t_1,t_2)=h_n(\infty,\infty) +\frac{p_n^2(-1,\infty,\infty)}{t_1}+\frac{p_n^2(1,\infty,\infty)}{t_2}+O\left(\frac{1}{t_1^2}+\frac{1}{t_2^2}\right),\label{tKKh}
\end{gather}
with
\begin{align*}
h_n(\infty,\infty)&=c\int_{-1}^1p_n^2(x,\infty,\infty) (1-x)^\alpha(1+x)^\beta\dd x,\\
&=\frac{2^{2n} n!(\alpha+1)_n(\beta+1)_n}{(\alpha+\beta+2n+1)[(\alpha+\beta+n+1)_n]^2(\alpha+\beta+2)_{n-1}},
\end{align*}
where the last equality arises from standard formulas for the Jacobi polynomials (see~\cite[equation~(10.8)]{EMOT}). This same reference provides the formulas required to evaluate
\begin{gather*}
p_n(1,\infty,\infty)=\frac{2^n(\alpha+1)_n}{(\alpha+\beta+n+1)_n}, \qquad p_n(-1,\infty,\infty)=(-1)^n\frac{2^n(\beta+1)_n}{(\alpha+\beta+n+1)_n}.
\end{gather*}
Hence, $y_n(t_1,t_2)$ is a~rational function of $(t_1,t_2)$. The Taylor expansion~\eqref{aKy} then follows from~\eqref{dKhy} and~\eqref{tKKh}. A~similar argument starting with~\eqref{SKh1} establishes~\eqref{aKz}.
\end{proof}

Using~\eqref{Ktheta}, \eqref{Komega}, \eqref{Krym1}, \eqref{Krzm1}, \eqref{Kra} and~\eqref{Kru}, the polynomial $K_n$ from~\eqref{K} involved in Laguerre's differential equation~\eqref{lae} can be explicitly computed in terms of $y_n$, $z_n$, $u_n$,~$v_n$,~$w_n$. Substituting~\eqref{uyz}, \eqref{vyz} and~\eqref{wyz} yields an expression for $K_n$ purely in terms of~$y_n$,~$z_n$ and their partial derivatives. As pointed out in Remarks~\ref{remark laguerre} and~\ref{remark jacobi}, for the Krall--Laguerre and the Krall--Jacobi cases, the first integrals of the Painlev\'e equations can be obtained by requiring $K_n$ to be divisible by $x$. By analogy, this suggests that now imposing $K_n$ to be divisible by $x^2-1$ will provide two first integrals of the Schlesinger system of partial differential equations satisfied by $y_n$ and $z_n$. This intuition turns out to be correct, as established in the following proposition.
\begin{Proposition} \label{proposition koornwinderfirstintegrals}
Let
\begin{gather}
r_n=t_1\frac{\partial}{\partial t_1} \log\frac{y_n}{z_n}, \qquad s_n=\frac{t_1}{z_n}\frac{\partial}{\partial t_1}\log y_n,\label{rs}
\end{gather}
then
\begin{gather}
r_n^2-4z_ns_n^2+\frac{\gamma_n^2z_n}{4(\beta+1)^2}=1,\label{SKFI1}\\
((\beta+1)(4s_n-r_n)+1)^2-4(\beta+1)^2y_ns_n^2+\frac{\gamma_n^2 y_n}{4}=(\alpha+1)^2,\label{SKFI2}
\end{gather}
are two first integrals of the system of partial differential equations~\eqref{PKy1}, \eqref{PKz1}, \eqref{PKy2}, \eqref{PKz2}, together with~\eqref{PK1} and~\eqref{PK2}. Consequently, the Laguerre differential equation~\eqref{lae} satisfied by the Koornwinder polynomials is explicitly given by
\begin{gather}
\bigl(x^2-1\bigr)\left(x^2-1+y_n\frac{1+x}{2}+z_n\frac{1-x}{2}\right)g_n''
+\biggl\{\bigl(x^2-1\bigr)\bigl((\alpha+\beta+2)x+\alpha-\beta\bigr) \nonumber\\
\qquad +\frac{1+x}{2}\bigl((\alpha+\beta+3)x+\alpha-\beta+1\bigr)y_n+\frac{1-x}{2} \bigl((\alpha+\beta+3)x+\alpha-\beta-1\bigr)z_n\biggr\}g_n'\nonumber\\
\qquad-\biggl\{\bigl(x^2-1\bigr)n(\alpha+\beta+n+1)-\frac{(\beta+1)(4x+y_n-z_n)}{2} (t_1\frac{\partial}{\partial t_1} \operatorname{log} z_n+1 ) \nonumber\\
\qquad+\frac{(\beta+1)}{2}\bigl(2(x+1)(y_n+z_n-4)+z_n(y_n-z_n+4)\bigr)\left(\frac{t_1}{z_n}\frac{\partial}{\partial t_1}\operatorname{log} y_n-\frac{\alpha+\beta+1}{4(\beta+1)}\right)\nonumber\\
\qquad+\frac{1+x}{2}\bigl(n(\alpha+\beta+n+1)y_n-\bigl((n-1)(\alpha+\beta)+n^2+n-1\bigr)z_n\bigr)\nonumber\\
\qquad +\frac{z_n}{8}\bigl((\alpha+\beta+1)(y_n-z_n)+4\bigl((2n-1)(\alpha+\beta)+2n^2+2n-1\bigr)\bigr)\biggr\}g_n\nonumber\\
\qquad=0.\label{Kode}
\end{gather}
In the limit as $t_1, t_2\to \infty$, this equation reduces to the standard differential equation satisfied by the Jacobi polynomials
\begin{gather}
\bigl(x^2-1\bigr)g_n''+((\alpha+\beta+2)x+\alpha-\beta)g_n'-n(\alpha+\beta+n+1)g_n=0.\label{Jade}
\end{gather}
\end{Proposition}
\begin{proof}
We start by establishing that the left-hand sides of~\eqref{SKFI1} and~\eqref{SKFI2} are first integrals of the system. From~\eqref{PKy1} and~\eqref{PKz1}, we deduce that
\begin{gather}
t_1\frac{\partial r_n}{\partial t_1}=-2z_ns_n^2+\frac{\gamma_n^2 z_n}{8(\beta+1)^2}.\label{drt1}
\end{gather}
Using~\eqref{PK1}, one computes that
\begin{align*}
\frac{\partial r_n}{\partial t_2}={}&\frac{\alpha+1}{(\beta+1)y_n}\left\{t_2\frac{\partial^2z_n}{\partial t_2^2}+\frac{\partial z_n}{\partial t_2}-\frac{t_2}{y_n}\frac{\partial y_n}{\partial t_2}\frac{\partial z_n}{\partial t_2}\right\}\\
&-\frac{t_1(\beta+1)}{t_2(\alpha+1)z_n}\left\{t_1\frac{\partial^2 y_n}{\partial t_1^2}+\frac{\partial y_n}{\partial t_1}-\frac{t_1}{z_n}\frac{\partial y_n}{\partial t_1}\frac{\partial z_n}{\partial t_1}\right\},
\end{align*}
hence, from~\eqref{PKy1}, \eqref{PKz2} and~\eqref{PK1}, we obtain
\begin{gather}
\frac{\partial r_n}{\partial t_2}=\frac{1}{2t_2(\alpha+1)(\beta+1)y_nz_n}\bigg\{t_2^2(\alpha+1)^2\left(\frac{\partial z_n}{\partial t_2}\right)^2-t_1^2(\beta+1)^2\left(\frac{\partial y_n}{\partial t_1}\right)^2\bigg\}=0.\label{drt2}
\end{gather}
Similarly, using~\eqref{PKy1},\eqref{PKz2} and~\eqref{PK1}, one computes that
\begin{gather}
32t_1(\beta+1)^2\frac{\partial s_n}{\partial t_1}=\bigl(\gamma_n^2-16(\beta+1)^2s_n^2\bigr)z_n,\label{dst1}\\
32t_2(\alpha+1)(\beta+1)\frac{\partial s_n}{\partial t_2}=\bigl(\gamma_n^2-16(\beta+1)^2s_n^2\bigr)y_n.\label{dst2}
\end{gather}
From~\eqref{drt1}, \eqref{drt2}, \eqref{dst1} and~\eqref{dst2}, using~\eqref{PK1}, \eqref{PK2} and~\eqref{rs}, it is easily checked that the partial derivatives with respect to $t_1$ and $t_2$ of the left-hand sides of~\eqref{SKFI1} and~\eqref{SKFI2} vanish.\looseness=1

{\samepage It remains to determine the constants on the right-hand sides of~\eqref{SKFI1} and~\eqref{SKFI2}. By direct substitution into~\eqref{PKz1}, \eqref{PKy2}, \eqref{PK1} and~\eqref{PK2}, the Taylor expansions for $y_n(t_1,t_2)$ in~\eqref{aKy} and~${z_n(t_1,t_2)}$ in~\eqref{aKz}, obtained in Lemma~\ref{lemma koornwindertaylor}, can be further evaluated, leading to~\eqref{aKy2} and~\eqref{aKz2} with $a$, $b$ as in~\eqref{ab}. Hence,}
{\samepage\begin{gather}
\lim_{t_1,t_2 \to \infty} t_1\frac{\partial}{\partial t_1}\log \frac{y_n}{z_n}=-\lim_{t_1,t_2\to \infty} t_1\frac{\partial}{\partial t_1} \log z_n=1,\label{ldy/z}\\
\lim_{t_1, t_2 \to \infty} \frac{t_1}{z_n}\frac{\partial}{\partial t_1} \log y_n=\frac{\alpha+\beta+1}{4(\beta+1)}, \label{ldyz}
\end{gather}
from which the constant values of the two first integrals follow immediately.}

Equation~\eqref{Kode} then follows from~\eqref{lae}, by computing $K_n$ in terms of $y_n$, $z_n$ and their partial derivatives, and using~\eqref{SKFI1} and~\eqref{SKFI2}. Finally, \eqref{ldy/z} and~\eqref{ldyz} show that in the limit as~${t_1,t_2\to \infty}$, \eqref{Kode} reduces to~\eqref{Jade}.
\end{proof}
\begin{Remark}
We note that the differential equation~\eqref{Kode} now possesses five regular singular points, including the point at infinity. In the special case $\alpha=\beta=0$, \eqref{Kode} was derived by Littlejohn~\cite{LLL2}, from the sixth-order differential operator of which the polynomials are eigenfunctions. Kiesel and Wimp~\cite{KW} obtained the differential equation in the general case, introducing new tools alongside the use of \textsc{Maple}. However, they overlooked the fact that the polynomial coefficient of $g_n$ is divisible by $x^2-1$, which leaves it as a~polynomial of degree $4$. As demonstrated by Proposition~\ref{proposition koornwinderfirstintegrals}, this divisibility property is equivalent to the integrability of the associated Schlesinger system. The explicit expressions in~\cite{KW,LLL2} are extremely cumbersome, highlighting the utility of the solutions $y_n$ and $z_n$ of the Schlesinger system in simplifying the problem.
\end{Remark}

In the next theorem, we obtain explicit formulas for $y_n$ and $z_n$, showing in particular that the solution of the Schlesinger system of partial differential equations in Theorem~\ref{theorem koornwinderschlesinger} is completely determined by the asymptotic behaviour~\eqref{aKy2}, \eqref{aKz2}, with $a$, $b$ as in~\eqref{ab}.
\begin{Theorem} \label{theorem koornwindersolution}
The rational functions $y_n(t_1,t_2,\alpha,\beta)$ and $z_n(t_1,t_2,\alpha,\beta)$ that characterize the Koornwinder polynomials are explicitly given by
\begin{gather}
y_n(t_1,t_2,\alpha,\beta)=\frac{4(\alpha+1)^2k_2t_2(k_1t_1+n)^2}{\sigma_n(t_1,t_2,\alpha,\beta)\tau_n(t_1,t_2,\alpha,\beta)},\qquad  n\geq 0,\label{Ky}
\end{gather}
where
\begin{gather}
\sigma_n(t_1,t_2,\alpha,\beta)=n^2(n-1)+n(\beta+n)k_1t_1+n(\alpha+n)k_2t_2
+(\alpha+\beta+n+1)k_1k_2t_1t_2,\nonumber\\
\tau_n(t_1,t_2,\alpha,\beta)=k_1k_2t_1t_2+(\alpha+n+1)k_1t_1+(\beta+n+1)k_2t_2
+n(\alpha+\beta+n+2),\label{Ksigmatau}
\end{gather}
and the coefficients $k_1$, $k_2$ are defined as follows:
\begin{gather}
k_1=\frac{n!(\beta+1)(\alpha+1)_n}{(\beta+1)_n(\alpha+\beta+2)_n},\qquad  k_2=\frac{n!(\alpha+1)(\beta+1)_n}{(\alpha+1)_n(\alpha+\beta+2)_n}.\label{k1k2}
\end{gather}
Moreover, we have
\begin{gather}
z_n(t_1,t_2,\alpha,\beta)=y_n(t_2,t_1,\beta,\alpha).\label{Kz}
\end{gather}
\end{Theorem}
\begin{proof}
The strategy of the proof is to first determine $r_n$ and $s_n$, from which the final result will follow. From~\eqref{drt1}, using~\eqref{SKFI1}, we obtain
\begin{gather}
2t_1\frac{\partial r_n}{\partial t_1}=1-r_n^2.\label{rt1}
\end{gather}
From~\eqref{drt2}, $r_n$ depends only on $t_1$. Hence, integrating~\eqref{rt1} gives
\begin{gather}
r_n=\frac{k_1t_1-n}{k_1t_1+n},\label{solr}
\end{gather}
with $k_1\neq 0$ being an arbitrary constant that depends on $n$, $\alpha$, $\beta$, still to be determined. We have chosen to write $r_n$ in this form so that the limit exists when $n=0$. Remembering the definition of $r_n$ \eqref{rs}, from the expansions~\eqref{aKy2} and~\eqref{aKz2}, we find that
\begin{gather}
k_1=\frac{4(\beta+1)^2}{(\alpha+\beta+n+1)b},\label{k1}
\end{gather}
which, using~\eqref{ab}, gives the expression for $k_1$ in~\eqref{k1k2}. Defining
$
\tilde{r}_n(t_1,t_2,\alpha,\beta)=r_n(t_2,t_1,\beta,\alpha)$,
it follows that $\tilde{r}_n$ depends only on $t_2$ and is given by
\begin{gather}
\tilde{r}_n=\frac{k_2t_2-n}{k_2t_2+n},\label{trt2}
\end{gather}
with $k_2$ obtained by permuting $\alpha$ with $\beta$ in $k_1$,
\begin{gather}
k_2=\frac{4(\alpha+1)^2}{(\alpha+\beta+n+1)a},\label{k2}
\end{gather}
thus yielding the expression for $k_2$ in~\eqref{k1k2}, using~\eqref{ab}.

Recalling again the definition of $r_n$ \eqref{rs}, we can integrate~\eqref{solr} to obtain the quotient
\begin{gather}
q_n(t_1,t_2,\alpha,\beta)=\frac{y_n(t_1,t_2,\alpha,\beta)}{z_n(t_1,t_2,\alpha,\beta)}=f(t_2,\alpha,\beta)\frac{(k_1t_1+n)^2}{t_1},\label{q}
\end{gather}
where $f(t_2,\alpha,\beta)$ is an arbitrary function of $t_2$ depending on $\alpha$, $\beta$. Recalling~\eqref{zy}, we have
\begin{gather*}
q_n(t_2,t_1,\beta,\alpha)=\frac{1}{q_n(t_1,t_2,\alpha,\beta)},
\end{gather*}
hence, we deduce that necessarily
\begin{gather*}
f(t_2,\alpha,\beta)=c(\alpha,\beta) \frac{t_2}{(k_2t_2+n)^2},
\end{gather*}
with $c(\alpha,\beta)$ a~constant still to be determined. Setting $t_1=t_2=t$ in~\eqref{q}, using~\eqref{aKy2}, \eqref{aKz2} and taking the limit as $t\to\infty$, we find that
\begin{gather*}
c(\alpha,\beta)=\frac{ak_2^2}{bk_1^2}=\frac{(\alpha+1)^2k_2}{(\beta+1)^2k_1},
\end{gather*}
where the last equality is obtained from~\eqref{k1} and~\eqref{k2}. Putting everything together leads to
\begin{gather}
q_n=\frac{(\alpha+1)^2k_2t_2(k_1t_1+n)^2}{(\beta+1)^2k_1t_1(k_2t_2+n)^2}=\frac{(\alpha+1)^2\bigl(\tilde{r}_n^2-1\bigr)}{(\beta+1)^2\bigl(r_n^2-1\bigr)}.\label{solq}
\end{gather}

Solving~\eqref{SKFI1} and~\eqref{SKFI2} for $y_n$ and $z_n$ and taking the quotient, gives
\begin{gather}
\bigl((\beta+1)(4s_n-r_n)+1\bigr)^2-(\alpha+1)^2=(\beta+1)^2\bigl(r_n^2-1\bigr)q_n,\label{qSKFI}
\end{gather}
which is a~quadratic equation for $s_n$ in terms of $r_n$ and $q_n$. Substituting~\eqref{solq} into~\eqref{qSKFI}, a~straightforward computation yields
\begin{gather}
s_n=\pm \frac{(\alpha+1)\tilde{r}_n}{4(\beta+1)} +\frac{r_n}{4}-\frac{1}{4(\beta+1)}.\label{sols}
\end{gather}
From the definition of $s_n$ \eqref{rs}, using~\eqref{PK1}, we observe that
\begin{gather}
s_n(t_2,t_1,\beta,\alpha)=\frac{\beta+1}{\alpha+1} s_n(t_1,t_2,\alpha,\beta).\label{pstalbe}
\end{gather}
We see at once that only the solution with the plus sign in~\eqref{sols} satisfies~\eqref{pstalbe}; hence, the other one is eliminated. Using~\eqref{solq} and~\eqref{qSKFI}, \eqref{SKFI2} can be written as
\begin{gather*}
(\alpha+1)^2\bigl(\tilde{r}_n^2-1\bigr)=\left(4(\beta+1)^2s_n^2-\frac{\gamma_n^2}{4}\right)y_n,
\end{gather*}
from which using~\eqref{sols} with the plus sign, we obtain
\begin{gather}
y_n=\frac{4(\alpha+1)^2\bigl(\tilde{r}_n^2-1\bigr)}{[(\alpha+1)\tilde{r}_n+(\beta+1)r_n+\gamma_n-1][(\alpha+1)\tilde{r}_n+(\beta+1)r_n-\gamma_n-1]}.\label{Kyrtr}
\end{gather}
Using~\eqref{Kga}, \eqref{solr} and~\eqref{trt2}, \eqref{Kyrtr} is equivalent to~\eqref{Ky}. As we already know from~\eqref{zy}, $z_n$ is obtained by permuting $t_1$ with $t_2$ and $\alpha$ with $\beta$ in~\eqref{Kyrtr}, establishing~\eqref{Kz}. Notice that~$\sigma_n$ and~$\tau_n$, as defined in~\eqref{Ksigmatau}, are invariant under this permutation. This completes the proof of Theorem~\ref{theorem koornwindersolution}.
\end{proof}
\begin{Corollary}
The coefficients $a_n(t_1,t_2)$ and $b_n(t_1,t_2)$ of the three-term recurrence relation defining the Koornwinder polynomials satisfy the following system of partial differential equations
\begin{gather}
\left(2(\alpha+1)t_2\frac{\partial}{\partial t_2}-2(\beta+1)t_1\frac{\partial}{\partial t_1}\right)a_n=(V_{-1}-V_1-(\alpha+\beta+2)T_1)a_n,\label{KTVam}\\
\left(2(\alpha+1)t_2\frac{\partial}{\partial t_2}-2(\beta+1)t_1\frac{\partial}{\partial t_1}\right)b_n=(V_{-1}-V_1-(\alpha+\beta+2)T_1)b_n,\label{KTVbm}
\end{gather}
and
\begin{gather}
\left(2(\alpha+1)t_2\frac{\partial}{\partial t_2}+2(\beta+1)t_1\frac{\partial}{\partial t_1}\right)a_n=(V_0-V_2+(\beta-\alpha)T_1
-(\alpha+\beta+3)T_2)a_n,\label{KTVap}\\
\left(2(\alpha+1)t_2\frac{\partial}{\partial t_2}+2(\beta+1)t_1\frac{\partial}{\partial t_1}\right)b_n=(V_0-V_2+(\beta-\alpha)T_1
-(\alpha+\beta+3)T_2)b_n,\label{KTVbp}
\end{gather}
with $T_1$, $T_2$, $V_{-1}$, $V_0$, $V_1$ and $V_2$ as in~\eqref{T1}, \eqref{T2}, \eqref{Vm1}, \eqref{V0}, \eqref{V1} and~\eqref{V2}.
\end{Corollary}
\begin{proof}
From~\eqref{Kga}, \eqref{SKa1} and~\eqref{SKa2}, we have
\begin{align}
4(\beta+1)t_1\frac{\partial a_n}{\partial t_1}&=a_n(\gamma_{n-1}z_{n-1}-\gamma_n z_n),\label{Kda1}\\
4(\alpha+1)t_2\frac{\partial a_n}{\partial t_2}&=a_n(\gamma_{n-1}y_{n-1}-\gamma_ny_n),\label{Kda2}
\end{align}
hence, using~\eqref{Kzmy}, we obtain
\begin{gather*}
4(\alpha+1)t_2\frac{\partial a_n}{\partial t_2}-4(\beta+1)t_1\frac{\partial a_n}{\partial t_1}
=2a_n((\alpha+\beta+2n-2)b_{n-1}-(\alpha+\beta+2n+2)b_n).
\end{gather*}
From~\eqref{uc}, \eqref{SKu1} and~\eqref{SKu2}, since $b_n=c_n-c_{n+1}$, we have
\begin{align}
8(\beta+1)t_1\frac{\partial b_n}{\partial t_1}&=(\gamma_n-1)w_n-(\gamma_{n+1}-1)w_{n+1},\label{Kdb1}\\
8(\alpha+1)t_2\frac{\partial b_n}{\partial t_2}&=(\gamma_n-1)v_n-(\gamma_{n+1}-1)v_{n+1}\label{Kdb2},
\end{align}
hence, using~\eqref{Kwmv}, we obtain
\begin{gather*}
8(\alpha+1)t_2\frac{\partial b_n}{\partial t_2}-8(\beta+1)t_1\frac{\partial b_n}{\partial t_1}\\
\qquad
=4\bigl(1-b_n^2+(\alpha+\beta+2n-1)a_n-(\alpha+\beta+2n+3)a_{n+1}\bigr).
\end{gather*}
Recalling the definitions~\eqref{T1}, \eqref{Vm1} and~\eqref{V1}, this establishes~\eqref{KTVam} and~\eqref{KTVbm}.

From~\eqref{Kzmy}, \eqref{Kda1} and~\eqref{Kda2}, we get
\begin{gather*}
2(\alpha+1)t_2\frac{\partial a_n}{\partial t_2}+2(\beta+1)t_1\frac{\partial a_n}{\partial t_1}\\
\qquad=a_n(\gamma_{n-1}y_{n-1}-\gamma_{n}y_n)
+a_n\bigl((\alpha+\beta+2n+2)b_n-(\alpha+\beta+2n-2)b_{n-1}\bigr),
\end{gather*}
hence, using~\eqref{Kyab}, we obtain
\begin{gather*}
2(\alpha+1)t_2\frac{\partial a_n}{\partial t_2}+2(\beta+1)t_1\frac{\partial a_n}{\partial t_1}\\
\qquad=2a_n\left(1-a_n+(b_{n-1}-b_n)\sum_{i=0}^{n-1}b_i\right)+a_n\bigl((\alpha+\beta+2n-3)\bigl(a_{n-1}+b_{n-1}^2\bigr)\\
\phantom{\qquad=}{}-(\alpha+\beta+2n+3)\bigl(a_{n+1}+b_n^2\bigr)\bigr)
+(\beta-\alpha)a_n(b_n-b_{n-1}).
\end{gather*}
From~\eqref{Kwmv}, \eqref{Kdb1} and~\eqref{Kdb2}, we get
\begin{align*}
&2(\alpha+1)t_2\frac{\partial b_n}{\partial t_2}+2(\beta+1)t_1\frac{\partial b_n}{\partial t_1}\\
& \qquad = \frac{(\gamma_n-1)v_n}{2}-\frac{(\gamma_{n+1}-1)v_{n+1}}{2}
 +\gamma_{n+1}a_{n+1}-\gamma_na_n+2a_n+b_n^2-1,
\end{align*}
hence, using~\eqref{Kvab}, we obtain
\begin{gather*}
2(\alpha+1)t_2\frac{\partial b_n}{\partial t_2}+2(\beta+1)t_1\frac{\partial b_n}{\partial t_1}\\
\qquad=b_n\bigl(1-b_n^2\bigr)
+a_n((\alpha+\beta+2n-2)b_{n-1}+(\alpha+\beta+2n-1)b_n)\nonumber\\
\phantom{\qquad=}{}-a_{n+1}((\alpha+\beta+2n+5)b_n+(\alpha+\beta+2n+4)b_{n+1})\nonumber\\
\phantom{\qquad=}{}
+(a_{n+1}-a_n)\left(\beta-\alpha-2\sum_{i=0}^{n-1}b_i\right).
\end{gather*}
Recalling~\eqref{T1}, \eqref{T2}, \eqref{V0} and~\eqref{V2}, this establishes~\eqref{KTVap} and~\eqref{KTVbp}, completing the proof.
\end{proof}

The Krall--Gegenbauer-type polynomials are a~special case of the Koornwinder polynomials, with $t_1=t_2=t$ and $\alpha=\beta$. Specializing the results of this section to this case, we obtain the following result, which establishes case (c)
of Theorem~\ref{theorem krallpainleve}.
\begin{Theorem} \label{theorem gegenbauerpainleve} %\label{theorem 5.3}
For $\alpha>-1$, define
\smash{$
\tilde{\tau}_n(t)= (\alpha+1)t+\frac{(2\alpha+2)_{n+1}}{n!}$}, $ n\geq 0$.
The Krall--Gegenbauer-type polynomials with the weight distribution given in~\eqref{wdKG} are completely characterized by the functions
\begin{gather}
x_0(t)=\frac{4(\alpha+1)}{(2\alpha+1)(t+2)}, \qquad x_n(t)=\frac{4(\alpha+1)^3(2\alpha+2)_{n-1}t}{n!\tilde{\tau}_{n-2}(t)\tilde{\tau}_n(t)},\qquad n\geq 1.\label{KGxn}
\end{gather}
Let $\gamma_n=2\alpha+2n+1$. The coefficients of the recurrence relation satisfied by these polynomials are given by
\begin{gather}
a_n=\frac{(\gamma_n-1)^2v_n(2x_n+v_n(x_n-1))}{4\gamma_{n-1}\gamma_nx_n^2},\qquad n\geq 1, \qquad b_n=0,\qquad  n\geq 0,\label{KGreab}
\end{gather}
where
\begin{gather}
v_n=\frac{\gamma_n x_n}{\gamma_n-1}-\frac{1}{(\gamma_n-1)(x_n-1)}\left(x_n+2t(\alpha+1)\dt{x_n}\right).\label{KGvy}
\end{gather}
The Laguerre equation~\eqref{lae} satisfied by these polynomials is given by
\begin{gather}
\bigl(x^2-1\bigr)\bigl(x^2-1+x_n\bigr)g''_n+2x\bigl((\alpha+1)\bigl(x^2-1\bigr)+(\alpha+2)x_n\bigr)g'_n-\biggl\{2(\alpha+1) \nonumber\\
\qquad\times \left(\frac{t}{x_n}\dt{x_n}+1\right)+\bigl(2(n-1)\alpha+n^2+n-1\bigr)x_n+n(2\alpha+n+1)\bigl(x^2-1\bigr)\biggr\}g_n\nonumber\\
\qquad\phantom{\times}{}=0.\label{KGde}
\end{gather}
As $t\to\infty$, the coefficients~\eqref{KGreab} reduce to those of the recurrence relation satisfied by the Gegenbauer polynomials, and~\eqref{KGde} becomes the standard differential equation for these polynomials
\begin{gather}
\bigl(1-x^2\bigr)g_n''-2(\alpha+1)x g_n'+n(2\alpha+n+1)g_n=0.\label{Gde}
\end{gather}
Moreover, the function $q_n(t)=1-x_n(t)$ solves an integrable case of the $P_{\rm V}$ equation~\eqref{P5}, with the parameters
\begin{gather}
a=\frac{\gamma_n^2}{8(\alpha+1)^2},\qquad b=-\frac{1}{8(\alpha+1)^2},\qquad c=0,\qquad d=0.\label{KGP5}
\end{gather}
\end{Theorem}
\begin{proof}
Since $\alpha=\beta$ and $t_1=t_2=t$, we define
\begin{gather}
x_n(t)=y_n(t,t,\alpha,\alpha)=z_n(t,t,\alpha,\alpha).\label{xn}
\end{gather}
Setting $t_1=t_2=t$ and $\beta=\alpha $ in~\eqref{Ksigmatau} gives
$
\sigma_n=(kt+n)\bigl((2\alpha+n+1)kt+n(n-1)\bigr)$, $ \tau_n=(kt+n)(kt+2\alpha+n+2)$,
where, from~\eqref{k1k2}, we have
\smash{$
k=k_1=k_2=\frac{n!(\alpha+1)}{(2\alpha+2)_n}$}.
Hence, using~\eqref{Ky}, it follows that
\[
x_n(t)=\frac{4(\alpha+1)^2kt}{[kt+2\alpha+n+2][(2\alpha+n+1)kt+n(n-1)]},\]
which yields~\eqref{KGxn}. From~\eqref{uva}, $v_n(t,t,\alpha,\alpha)=-w_n(t,t,\alpha,\alpha)$. Thus, using~\eqref{xn}, we obtain from~\eqref{Kru} that $u_n=0$, which implies~\eqref{KGreab} by virtue of~\eqref{Krb} and~\eqref{Kra}.

Recalling from~\eqref{zy} that $y_n(t_1,t_2,\alpha,\beta)=z_n(t_2,t_1,\beta,\alpha)$, we have
\begin{gather}
\frac{\partial y_n}{\partial t_2}(t,t,\alpha,\alpha)=\frac{\partial z_n}{\partial t_1}(t,t,\alpha,\alpha).\label{pyt2pzt1}
\end{gather}
From~\eqref{xn}, we have
\smash{$
\dt{x_n}(t)=\frac{\partial y_n}{\partial t_1}(t,t,\alpha,\alpha)+\frac{\partial y_n}{\partial t_2}(t,t,\alpha,\alpha)$},
and from~\eqref{PK2} with $\alpha=\beta$ and $t_1=t_2=t$, combined with~\eqref{pyt2pzt1}, we also have
\begin{gather*}
x_n^2=2t(\alpha+1)\left((x_n-2)\frac{\partial y_n}{\partial t_1}(t,t,\alpha,\alpha)-x_n\frac{\partial y_n}{\partial t_2}(t,t,\alpha,\alpha)\right).
\end{gather*}
From the last two equations, we deduce
\begin{align}
\frac{\partial y_n}{\partial t_1}(t,t,\alpha,\alpha)&=\frac{x_n}{2(x_n-1)}\left(\dt{x_n}+\frac{x_n}{2(\alpha+1)t}\right),\label{pyt1x}\\
\frac{\partial y_n}{\partial t_2}(t,t,\alpha,\alpha)&=\frac{1}{2(x_n-1)}\left((x_n-2)\dt{x_n}-\frac{x_n^2}{2(\alpha+1)t}\right)\label{pyt2x}.
\end{align}

Using~\eqref{pyt2pzt1}, \eqref{pyt1x} and~\eqref{pyt2x}, one can verify that~\eqref{vyz} and~\eqref{Kode} reduce to~\eqref{KGvy} and~\eqref{KGde}, respectively. Moreover, since \smash{$\lim_{t\to\infty} x_n(t)=0$} and \smash{$\lim_{t\to\infty}\frac{t}{x_n(t)}\dt{x_n(t)}=-1$}, equation~\eqref{Gde} holds. Using the same equations, the two first integrals~\eqref{SKFI1} and~\eqref{SKFI2} become identical, reducing to
\begin{gather*}
\frac{t^2}{x_n^2(1-x_n)}\left(\dt{x_n}\right)^2-\frac{1}{4(\alpha+1)^2(1-x_n)}+\frac{\gamma_n^2x_n+1}{4(\alpha+1)^2}=1.
\end{gather*}
By setting $x_n=1-q_n$ and recalling~\eqref{FIP5}, one recognizes this expression as a~first integral of the integrable case of the $P_{\rm V}$ equation with the parameters as announced in~\eqref{KGP5}. This completes the proof of Theorem~\ref{theorem gegenbauerpainleve}.
\end{proof}
\begin{Remark}
In the special case $\alpha=\beta=0$ and $t_1=t_2=t$, Krall--Gegenbauer-type polynomials are usually called Legendre-type polynomials. In this case, \eqref{KGde} was obtained by Littlejohn and Shore~\cite{LLS}, using the fourth-order differential operator of which these polynomials are eigenfunctions. The Legendre-type polynomials are part of the full solution of Krall's problem for an operator of order four~\cite{K2}.
\end{Remark}

\section{Koornwinder's formulas for Krall-type polynomials}\label{sec6}
In~\cite{Ko}, Koornwinder defines his polynomials by an explicit differentiation formula. In this section, we derive his formula from our results. Using this formula, Koornwinder was able to express the polynomials in terms of $_4F_3$ hypergeometric functions, or $_3F_2$ hypergeometric functions when one mass point is zero. We begin our analysis with the simpler case of Krall--Laguerre-type polynomials.
\begin{Theorem} \label{theorem klaguerredformula}
Let $\alpha>-1$. The Krall--Laguerre-type polynomials $p_n(x, t,\alpha)$, which are orthogonal with respect to the weight distribution~\eqref{wdKLa}, are given in terms of the generalized Laguerre polynomials $p_n^{\alpha}(x)$, normalized to be monic, by the formula
\begin{gather}
p_n(x,t,\alpha)=\left(\frac{(\alpha+1)_n}{n!t+n(\alpha+2)_{n-1}}\frac{\mathrm{d}}{\dd x}+1\right)p_n^{\alpha}(x).\label{KLK}
\end{gather}
\end{Theorem}
\begin{proof}
Set
\begin{gather}
g_n=\left(U_n\frac{\mathrm{d}}{\dd x}+1\right)p_n^\alpha,\label{dKLa}
\end{gather}
where $U_n$ is independent of $x$ and remains to be determined. From~\eqref{dKLa} and the second-order differential equation~\eqref{Lde} satisfied by the generalized Laguerre polynomials, one finds that $g_n$ satisfies the second-order linear differential equation
\begin{gather}
x(-x+y_n)g''_n+\bigl((x-y_n)(x-\alpha-2)+x\bigr)g'_n\nonumber\\
\qquad+\left\{(n-1)y_n
+\frac{U_n\bigl(\alpha+1+n(1-x)\bigr)-nx}{U_n+1}\right\}g_n=0, \label{ddKLa}
\end{gather}
with
\begin{gather}
y_n=\frac{(\alpha+1)U_n-nU_n^2}{U_n+1}.\label{lay}
\end{gather}
Forcing equation~\eqref{ddKLa} to coincide with~\eqref{KLde} yields
\begin{gather}
U_n=-\frac{(\alpha+1)\bigl(t\dt{y_n}+y_n\bigr)+y_n^2}{(\alpha+1)\bigl(t\dt{y_n}-y_n\bigr)+y_n(y_n-2n)}.\label{laq}
\end{gather}
By choosing $y_n(t)$ in~\eqref{laq} to be the solution of the integrable case of $P_{\rm III}$ that characterizes the Krall--Laguerre-type polynomials, we find that~\eqref{lay} is identically satisfied by virtue of~\eqref{FIPy}. Thus, with this definition of $U_n$, $g_n$ in~\eqref{dKLa} solves~\eqref{KLde}. Since this equation has a~unique polynomial solution (up to a~multiplicative constant), and since $g_n$ is a~monic polynomial in $x$, it must coincide with $p_n(x,t,\alpha)$. Substituting~\eqref{SFIPy} into~\eqref{laq}, with $k$ as in~\eqref{kSFIPy}, yields~\eqref{KLK}, concluding the proof.
\end{proof}
\begin{Remark}
Using~\eqref{KLK}, J. Koekoek and R. Koekoek~\cite{KK1} showed that for $t>0, t\neq \infty$, Krall--Laguerre-type polynomials solve Krall's problem~\cite{K1} if and only if $\alpha$ is a~nonnegative integer. In~\cite[Section 3.1]{GHH}, it is shown that~\eqref{KLK} can also be derived via the Darboux transformation method applied to semi-infinite Jacobi matrices.
\end{Remark}
\begin{Theorem}
Let $\alpha, \beta>-1$. The Koornwinder polynomials $p_n(x, t_1, t_2, \alpha, \beta)$, which are orthogonal with respect to the weight distribution~\eqref{wdK}, are given in terms of the Jacobi polynomials \smash{$p_n^{(\alpha, \beta)}(x)$}, normalized to be monic on $[-1,1]$, by the formula
\begin{align}
p_n(x, t_1, t_2, \alpha, \beta)={}&\bigg\{\Big[B_n(1-x)-A_n(1+x)\Big]\frac{\mathrm{d}}{\dd x}+ \frac{(\alpha+\beta+n+1)A_nB_n}{(\alpha+1)(\beta+1)}\bigg\}\nonumber\\
&\times\frac{p_n^{(\alpha,\beta)}(x)}{\sigma_n(t_1,t_2,\alpha,\beta)},\label{Kd}
\end{align}
where $\sigma_n(t_1,t_2,\alpha,\beta)$ is given by~\eqref{Ksigmatau}, and $A_n$, $B_n$ are defined by
\begin{gather}
A_n=(\alpha+1)(k_1t_1+n),\qquad  B_n=(\beta+1)(k_2t_2+n),\label{cKd}
\end{gather}
with $k_1$, $k_2$ given by~\eqref{k1k2}.
\end{Theorem}
\begin{proof}
Set
\begin{gather}
g_n=\biggl\{\big[V_n(1-x)-U_n(1+x)\big]\frac{\mathrm{d}}{\dd x}+n(U_n+V_n)+1\biggr\}p_n^{(\alpha,\beta)},\label{dK}
\end{gather}
so that $g_n$ is a~monic polynomial of degree $n$, where $U_n$ and $V_n$ are coefficients independent of~$x$ to be determined. Then, defining
\begin{align}
y_n&=\frac{4U_n\bigl((\alpha+1)(1+nV_n)-n(\beta+n)U_n\bigr)}{(\alpha+\beta+2n+1)(U_n+V_n)+1},\label{yuv}\\
z_n&=\frac{4V_n\bigl((\beta+1)(1+nU_n)-n(\alpha+n)V_n\bigr)}{(\alpha+\beta+2n+1)(U_n+V_n)+1},\label{zuv}
\end{align}
and using the differential equation~\eqref{Jade} satisfied by the Jacobi polynomials, a~direct computation shows that $g_n$ satisfies a~second-order linear differential equation. The coefficients of $g_n''$ and $g_n'$ are formally identical to those in~\eqref{Kode}, and the coefficient of $g_n$ is given by
\begin{gather}
\bigl(1-x^2\bigr)n(\alpha+\beta+n+1)+\frac{2n(\alpha+\beta+n+1)}{(\alpha+\beta+2n+1)(U_n+V_n)+1}\nonumber\\
\qquad \times\{(1+x)\bigl((n-1)(\beta+n)U_n^2-(\alpha+2)U_n\bigr)+(1-x)\bigl((n-1)(\alpha+n)V_n^2-(\beta+2)V_n\bigr)\nonumber\\
\phantom{\qquad \times}{}-(n-1)\bigl((1+x)\alpha+(1-x)\beta+4\bigr)U_nV_n\}.\label{gde}
\end{gather}
Equating~\eqref{gde} with the coefficient of $g_n$ in~\eqref{Kode} where, recalling the definition of $r_n$, $s_n$ in~\eqref{rs}, we set
\begin{gather*}
t_1\frac{\partial}{\partial t_1}\log z_n=z_ns_n-r_n, \qquad \frac{t_1}{z_n}\frac{\partial}{\partial t_1}\log y_n=s_n,
\end{gather*}
and substituting~\eqref{yuv} and~\eqref{zuv} for $y_n$ and $z_n$, allows one to solve for $r_n$ and $s_n$,
\begin{align*}
r_n&=\frac{(\beta+1)(1+nU_n)-n(2\alpha+\beta+2n+1)V_n}{(\beta+1)(n(U_n+V_n)+1)},\\
s_n&=\frac{\alpha+\beta+1-n(\alpha+\beta+2n+1)(U_n+V_n)}{4(\beta+1)(n(U_n+V_n)+1)},
\end{align*}
which yields
\begin{gather}
U_n=\frac{(\beta+1)(r_n-4s_n)+\alpha}{n\bigl(4(\beta+1)s_n+\alpha+\beta+2n+1\bigr)},\qquad
V_n=\frac{(\beta+1)(1-r_n)}{n\bigl(4(\beta+1)s_n+\alpha+\beta+2n+1\bigr)}.\label{guv}
\end{gather}
Substituting~\eqref{sols} (with the plus sign), \eqref{solr} and~\eqref{trt2} into~\eqref{guv}, we find
\begin{gather}
U_n=\frac{(\alpha+1)(k_1t_1+n)}{\sigma_n(t_1,t_2,\alpha,\beta)},\qquad
V_n=\frac{(\beta+1)(k_2t_2+n)}{\sigma_n(t_1,t_2,\alpha,\beta)}, \label{guvt1t2}
\end{gather}
where $\sigma_n(t_1,t_2,\alpha,\beta)$ is given in~\eqref{Ksigmatau} and $k_1$, $k_2$ are defined as in~\eqref{k1k2}. This matches the differentiation formula~\eqref{Kd} with the coefficients $A_n$, $B_n$ given in~\eqref{cKd}. It remains to check that substituting~\eqref{guvt1t2} into~\eqref{yuv} and~\eqref{zuv} yields formulas for $y_n$ and $z_n$ that agree with those found in~\eqref{Ky} and~\eqref{Kz}, which is easily verified. By the same argument as in Theorem~\ref{theorem klaguerredformula}, this implies that~$g_n$, defined in~\eqref{dK} with $U_n$, $V_n$ as in~\eqref{guvt1t2}, is equal to $p_n(x,t_1,t_2,\alpha,\beta)$, which concludes the proof.
\end{proof}

\begin{Remark}
Using~\eqref{Kd}, J.~Koekoek and R.~Koekoek~\cite{KK2} showed that for $t_1, t_2>0$, Koornwinder's polynomials solve Krall's problem~\cite{K1} if and only if $\alpha$ and $\beta$ are nonnegative integers when $t_1,t_2\neq \infty$, and $\alpha$ or $\beta$ is a~nonnegative integer when $t_1=\infty$ or $t_2=\infty$.
\end{Remark}

\subsection*{Acknowledgements}
The author wishes to thank M. Bertola for bringing reference~\cite{BEH} to his attention, and A.P. Magnus for pointing out reference~\cite{Re}. The author is also grateful to the anonymous referees for their careful reading and helpful suggestions, which improved the presentation of the manuscript.

\pdfbookmark[1]{References}{ref}
\LastPageEnding

\end{document}